\documentclass[11pt,a4paper]{article}
\usepackage{pict2e}
\usepackage{booktabs}
\usepackage{color}
\usepackage{epsfig}
\usepackage{epstopdf}
\usepackage{longtable}
\usepackage[T1]{fontenc}
\usepackage{amsmath}
\usepackage{cases}
\usepackage{geometry}
\usepackage{amsbsy,latexsym,amsfonts, epsfig, color, authblk, amssymb, graphics, bm}
\usepackage{epsf,slidesec,epic,eepic}
\usepackage{fancybox}
\usepackage{fancyhdr}
\usepackage{setspace}
\usepackage{nccmath}
\usepackage{diagbox}
\usepackage[colorlinks, citecolor=blue]{hyperref}

\newtheorem{theorem}{Theorem}[section]
\newtheorem{lemma}{Lemma}[section]

\newtheorem{definition}{Definition}[section]

\newtheorem{proposition}{Proposition}[section]

\newtheorem{remark}{Remark}[section]

\newenvironment{proof}{{\noindent \bf Proof:}}{\hfill$\Box$\medskip}

\definecolor{lred}{rgb}{1,0.8,0.8}
\definecolor{lblue}{rgb}{0.8,0.8,1}
\definecolor{dred}{rgb}{0.6,0,0}
\definecolor{dblue}{rgb}{0,0,0.5}
\definecolor{dgreen}{rgb}{0,0.5,0.5}

 \title{Linear convergence of the generalized PPA and several splitting methods for the composite inclusion problem}

 \author{Li Shen\footnote{Department of Mathematics, South China University of Technology, Tianhe District of Guangzhou, 510641, China (shen.li@mail.scut.edu.cn).}
  \ \ {\rm and}\ \ Shaohua Pan\footnote{Corresponding author. Department of Mathematics, South China University of Technology, Tianhe District of Guangzhou, China(shhpan@scut.edu.cn).}}
 \date{August 16, 2015; Revised: April 6, 2016}

\begin{document}

 \maketitle

 \begin{abstract}
  For the inclusion problem involving two maximal monotone operators,
  under the metric subregularity of the composite operator, we derive the linear convergence
  of the generalized proximal point algorithm and several splitting algorithms, which include
  the over-relaxed forward-backward splitting algorithm, the generalized Douglas-Rachford
  splitting algorithm and Davis' three-operator splitting algorithm. To the best of our knowledge,
  this linear convergence condition is weaker than the existing ones that almost all require
  the strong monotonicity of the composite operator. Withal, we give some sufficient conditions
  to ensure the metric subregularity of the composite operator. At last, the preliminary numerical
  performances on some toy examples  support the theoretical results.

  \end{abstract}

  \section{Introduction}\label{sec1}
   Let $\mathbb{X},\mathbb{Y}$ and $\mathbb{Z}$ be the finite dimensional linear spaces endowed
   with the inner product $\langle\cdot,\cdot\rangle$ and its induced norm $\|\cdot\|$.
   Given the maximal monotone operators $\mathcal{A},\mathcal{B}\!:\mathbb{Z}\rightrightarrows\mathbb{Z}$
   and $\vartheta$-cocoercive operator $\mathcal{C}\!:\mathbb{Z}\rightrightarrows\mathbb{Z}$,
   we focus on the composite operator inclusion problem
   \begin{equation}\label{composite-inclusion}
    0 \in \mathcal{F}z\quad\ {\rm with}\ \ \mathcal{F}\!:=\mathcal{A}+\mathcal{B}+\mathcal{C}.
   \end{equation}
   We are interested in the case that one of
   $\mathcal{A},\mathcal{B}$ is single valued and Lipschitz continuous.
   Unless otherwise stated, we always assume that $\mathcal{F}^{-1}(0)\ne \emptyset$
   for problems \eqref{composite-inclusion}.

   \medskip

  The inclusion problem \eqref{composite-inclusion} has many applications such as the variational inequality problem \cite{PP90},
  the problem of finding a common point of closed convex sets \cite{BB95,BCL04} and
  covers many classes of convex optimization problems.
  Specifically,  we consider the following unconstrained nonsmooth composite convex minimization problem with the form that
 \begin{align}\label{convex-composite-prob}
 \min f(x)+g(\mathcal{D}x)+h(x)
 \end{align}
 where $f\!:\!\mathbb{Z}\to [-\infty,+\infty]$ and $g\!:\!\mathbb{Y}\to [-\infty,+\infty]$
 are low semicontinuous convex functions, $g\!:\!\mathbb{Z}\to (-\infty,+\infty)$ is continuous
 differentiable and gradient Lipschitz convex function and $\mathcal{D}:\mathbb{Z}\to\mathbb{Y}$ is linear operators.
 Involved in the dual variable $y\in \mathbb{Y}$, it is easy to check that solving the optimization
 problem \eqref{convex-composite-prob} is equivalent to solve the inclusion
 \begin{equation}\label{equi-convex-composite}
 (0,0) \in (\mathcal{A}+\mathcal{B}+\mathcal{C})(x,y)
 \end{equation}
 with $\mathcal{C}(x,y) = \big(\nabla h(x),0\big)$ being the cocoercive operator,
 $\mathcal{A}(x,y) = (\mathcal{D}^{*}y, -\mathcal{D}x)$ being the single valued, Lipschitz continuous operator and
 $\mathcal{B}(x,y) = \big(\partial f(x),\partial g^*(y)\big)$ being the maximal monotone operator
 where $\partial g^*$ is the conjugate function of $g$, which indicates that the composite convex
 problem \eqref{convex-composite-prob} can be reformulated as a special case
 of \eqref{composite-inclusion} with the above specified operators $\mathcal{A},\mathcal{B},\mathcal{C}$. Furthermore, when $g$
 is also continuous differentiable and gradient Lipschitz, the problem \eqref{convex-composite-prob}
 can be directly represented as the form of \eqref{composite-inclusion} with $\mathcal{A}=\partial f$,
 $\mathcal{B}=\mathcal{D}^{*}\circ\nabla g\circ\mathcal{D}$ and $\mathcal{C}=\nabla h$.
  Additionally, the inclusion problem \eqref{composite-inclusion} is highly related to the linearly
  constrained two-block separable convex minimization problem
  which has many applications as listed in \cite{BPCPE12} and takes the following form that
  \begin{align*}
   &\min_{u\in\mathbb{X},v\in\mathbb{Y}} f(u)+g(v)\nonumber\\
   &\quad {\rm s.t.}\ \ \mathcal{G}u+\mathcal{H}v=c.
  \end{align*}
  In addition, it has a dual problem falling into framework of \eqref{convex-composite-prob} with the following form
  \begin{equation*}
   \min_{z\in\mathbb{Z}}\ \Phi(z):=f^*(\mathcal{G}^*z)+g^*(\mathcal{H}^*z)-\langle c,z\rangle,
  \end{equation*}
  which also can be reformulated as a special case of inclusion \eqref{composite-inclusion}.
  Here $f\!:\mathbb{X}\to(-\infty,+\infty]$ and $g\!:\mathbb{Y}\to(-\infty,+\infty]$
  are closed proper convex functions whose conjugate function are written as $f^*$ and $g^*$,
  respectively, $\mathcal{G}\!:\mathbb{X}\to\mathbb{Z}$ and $\mathcal{H}\!:\mathbb{Y}\!\to\!\mathbb{Z}$
  are linear operators whose adjoint are $\mathcal{G}^*$ and $\mathcal{H}^*$, respectively,
  and $c\in\mathbb{Z}$ is a vector.

  \medskip

  Let $\mathcal{J}_{\gamma\mathcal{T}}$ denote the resolvent of an operator $\mathcal{T}$ of index $\gamma\!>\!0$,
  i.e., $\mathcal{J}_{\gamma\mathcal{T}}\!=\!(\mathcal{I}+\!\gamma\mathcal{T})^{-1}$. If by chance the calculation
  of $\mathcal{J}_{\gamma \mathcal{F}}$ is easy, then the classical proximal point algorithm
  \cite{Roc76} or its over-relaxed version (the generalized PPA \cite{EB92}) is a desirable solver
  for the inclusion problem  \eqref{composite-inclusion}. However,
  in practice the estimation of $\mathcal{J}_{\gamma \mathcal{F}}$ is usually much more difficult
  than that of $\mathcal{J}_{\gamma \mathcal{A}}$, $\mathcal{J}_{\gamma \mathcal{B}}$ and $\mathcal{J}_{\gamma \mathcal{C}}$.
  Motivated by this, Davis and Yin \cite{Davis15} proposed a three operator splitting method
  for the inclusion \eqref{composite-inclusion} by using $\mathcal{J}_{\gamma\mathcal{A}}$ and
  $\mathcal{J}_{\gamma\mathcal{B}}$. Moreover, when one of $\mathcal{B}$ and $\mathcal{C}$ vanishes,
  several two operator splitting algorithms, including the forward-backward splitting (FBS)
  algorithm \cite{Passty79,Gabay83,CRoc97,CR06}, the Peaceman-Rachford splitting (PRS) algorithm
  \cite{LM79,PR54}, and the Douglas-Rachford splitting (DRS) algorithm \cite{LM79,DR56},
  have been developed by using $\mathcal{J}_{\gamma\mathcal{A}}$ or/and $\mathcal{J}_{\gamma\mathcal{B}}$, $\mathcal{J}_{\gamma\mathcal{C}}$.
  The convergence of these splitting algorithms have been well studied. In view of the strong assumption
  required by the linear convergence, some authors recently focus on the iteration complexity of these algorithms
  \cite{HYuan14,Davis151,LFP15}. By contrast, the study for
  their convergence rate is quite deficient. To the best of our knowledge, several existing linear
  convergence rate results (see \cite{LM79,CRoc97,Giselsson15,Davis152,Davis15}) all require
  the strong monotonicity of one of the operators $\mathcal{A}$, $\mathcal{B}$ and $\mathcal{C}$
  and single valued and Lipschitz continuous property of $\mathcal{B}$. It is well known that
  the strong monotonicity assumption of $\mathcal{A},\mathcal{B}$ or $\mathcal{C}$ is too stringent.
  Recently, Liang et al. \cite{LFP15} and Bauschke et al. \cite{BNP15} establish the local
  linear convergence rate of substantial splitting algorithms based on the Krasnosel'ski\u{\i}-Mann
  fixed point iteration\cite{KM55,M53} scheme with the metric subregularity assumption and
  the (bounded) linear regularity assumption  on the fixed point operator at a point of its graph,
  respectively. The condition used by Liang et al. \cite{LFP15} is shown to be equivalent to the metric subregularity
  condition of $\mathcal{F}$ at a point $(x^*,0)\in {\rm gph}\,\mathcal{F}$ for the generalized PPA algorithm
  and over-relaxed FBS algorithm according to lemma \ref{relation-AB} in the following.

  \medskip

  The main contribution of this paper is to derive the linear convergence rate of the above
  several splitting algorithms and Davis' splitting algorithm \cite{Davis15} under the metric
  subregularity of the operator $\mathcal{F}$ at a point $(\overline{x},0)$ of its graph.
  In addition, as will be shown in the section \ref{sec3}, the metric subregularity of an operator at a point
  of its graph is weaker than some existing regularization conditions such as the strongly monotone
  and the projective type error bound \cite{Tseng95} on the fixed point operator.

  \section{Preliminaries}\label{sec2}

  This section recalls some necessary concepts and lemmas that will be used in the subsequent analysis.
  Firstly, we introduce some concepts associated to an operator
  $\mathcal{T}\!:\mathbb{X}\rightrightarrows\mathbb{Y}$, for which we make no difference from its graph
  ${\rm gph}\,\mathcal{T}:=\left\{(x,y)\in\mathbb{X}\times\mathbb{Y}\ |\ y\in\mathcal{T}(x)\right\}$.
  The domain and range of an operator $\mathcal{T}\!:\mathbb{X}\rightrightarrows\mathbb{Y}$
  are respectively defined as
  \[
    {\rm dom}\,\mathcal{T}=\big\{x\in\mathbb{X}\ |\ \mathcal{T}x\ne \emptyset\big\}\ \ {\rm and}\ \
    {\rm ran}\,\mathcal{T}=\big\{y\in\mathbb{Y}\ |\ \exists x\in\mathbb{X}\ {\rm such\  that}\ (x,y)\in\mathcal{T}\big\}.
  \]
  The inverse of $\mathcal{T}$ is given by $\mathcal{T}^{-1}:=\{(y,x)\in\mathbb{Y}\times\mathbb{X}\ |\ (x,y)\in\mathcal{T}\}$.
  For any $c\in\mathbb{R}$, we let $c\mathcal{T}=\{(x,cy)\ |\ (x,y)\in\mathcal{T}\}$,
  and if $\mathcal{G}$ and $\mathcal{H}$ are any operators from $\mathbb{X}$ to $\mathbb{Y}$,
  we let
  \[
    \mathcal{G}+\mathcal{H}=\big\{(x,y+z)\ |\ (x,y)\in\mathcal{G},\,(x,z)\in\mathcal{H}\big\}.
  \]
  An operator $\mathcal{T}\!:\mathbb{X}\rightrightarrows\mathbb{X}$ is said to be firmly nonexpansive if
  \(
    \langle x-y,u-v\rangle\ge \|u-v\|^2
  \)
  for $(x,u),(y,v)\in\mathcal{T}$. Moreover, $\mathcal{T}$ is said to be nonexpansive if
  $\|\mathcal{T}u-\mathcal{T}v\|\le\|x-y\|$. By \cite[Proposition 5.14]{BC11}
  we have the following result for a nonexpansive operator.
  \begin{lemma}\label{main-lemma1}
   Let $\mathcal{T}\!:\mathbb{X}\to\mathbb{X}$ be a nonexpansive operator with ${\rm Fix}\,\mathcal{T}\ne \emptyset$,
   and $\{\mu_k\}$ be a sequence in $[0,1]$ satisfying $\sum_{k=0}^{\infty}\mu_k(1-\mu_k)=+\infty$. Let $\{w^k\}$
   be generated by
  \begin{align}\label{KM}
    w^{k+1} = w^{k} + \mu_k(\mathcal{T}w^k-w^{k})\quad{\rm with}\ \ w^0\in\mathbb{X}.
  \end{align}
  Then, the sequences $\{w^k\}$ and $\{\mathcal{T}w^{k}\}$ converge to a point in ${\rm Fix}\,\mathcal{T}$
  and for any $w \in {\rm Fix}\,\mathcal{T}$,
  \[
    \|w^{k+1}-w\|^2 \le \|w^{k}-w\|^2-\mu_k(1-\mu_k)\|\mathcal{T}w^{k}-w^k\|^2
    \quad{\rm for\ all}\ \ k\in\mathbb{N}.
  \]
 \end{lemma}

  Next we recall from the monograph \cite{BC11} the concept of the $\alpha$-averaged operator.
  \begin{definition}\label{averaged}
   Let $D$ be a nonempty subset of $\mathbb{X}$, $\mathcal{T}\!:D\to\mathbb{X}$ be a nonexpansive operator,
   and $\alpha\in(0,1)$ be a constant. Then the operator $\mathcal{T}$ is said to be $\alpha$-averaged
   if there exists a nonexpansive operator $\mathcal{R}\!:D\to\mathbb{X}$ such that
   $\mathcal{T}=(1-\alpha)\mathcal{I}d+\alpha\mathcal{R}$.
  \end{definition}
   By \cite[Prop. 5.15]{BC11} we have the following result for an $\alpha$-averaged operator
   $\mathcal{T}\!:\mathbb{X}\to\mathbb{X}$.
  \begin{lemma}\label{main-lemma2}
   Let $\mathcal{T}\!:\mathbb{X}\to\mathbb{X}$ be an $\alpha$-averaged operator of $\alpha\in(0,1)$
   with ${\rm Fix}\,\mathcal{T}\ne \emptyset$, and $\{\mu_k\}\subseteq [0,\frac{1}{\alpha}]$ be a sequence
   satisfying $\sum_{k=0}^{\infty}\mu_k(\frac{1}{\alpha}-\mu_k)=+\infty$. Let $\{w^k\}$
   be generated by \eqref{KM}. Then, $\{w^k\}$ and $\{\mathcal{T}w^{k}\}$ converge to a point in ${\rm Fix}\,\mathcal{T}$
  and for any $w \in {\rm Fix}\,\mathcal{T}$,
  \[
    \|w^{k+1}-w\|^2 \le \|w^{k}-w\|^2-\mu_k(\alpha^{-1}-\mu_k)\|\mathcal{T}w^{k}-w^k\|^2
    \quad{\rm for\ all}\ \ k\in\mathbb{N}.
  \]
  \end{lemma}

   The following definition is about the metric subregular \cite{DR09} of
  $\mathcal{T}$ at $(\overline{x},\overline{y}) \in{\rm gph}\,\mathcal{T}$.

  \begin{definition}\label{metric-subregular}
   An operator $\mathcal{T}\!:\mathbb{X}\rightrightarrows\mathbb{Y}$ is metrically subregular
   at $(\overline{x},\overline{y})\in{\rm gph}\mathcal{T}$ with constant $\kappa>0$
   if there exists a neighborhood $U$ of $\overline{x}$ such that
  \begin{equation*}
    {\rm dist}\big(x,\mathcal{T}^{-1}(\overline{y})\big) \le \kappa{\rm dist}\big(\overline{y},\mathcal{T}x\big)\quad {\rm \ for \ all}\ x\in U.
  \end{equation*}
  \end{definition}

 \section{Linear convergence of several splitting algorithms}\label{sec3}
  In the first four subsection, we derive the linear convergence rate of the generalized PPA
  with both $\mathcal{B}$ and $\mathcal{C}$ vanishing, the over-relaxed FBS and the generalized DRS algorithm
  for problem \eqref{composite-inclusion} with the corresponding $\mathcal{B}$ and $\mathcal{C}$ vanishing,
  and Davis-Yin's three-operator splitting algorithm
  for problem \eqref{composite-inclusion} under assumption that $\mathcal{F}$ is
  metric subregular at a point $(z^*,0)\in {\rm gph}\,\mathcal{F}$. In the last subsection,
  we will discuss the equivalence on metric subregularity
  condition between the inclusion operator $\mathcal{F}$ in \eqref{composite-inclusion} and
  fixed point operator in \cite{LFP15} for generalized PPA and over-relaxed FBS algorithm.
  Some sufficient conditions are also given to ensure the metric subregular of $\mathcal{F}$ at $(z^*,0)\in {\rm gph}\,\mathcal{F}$.


  \subsection{Linear convergence of generalized PPA}\label{subsec3}

   The generalized PPA \cite{EB92} for problem \eqref{composite-inclusion} with $\mathcal{B}$ and $\mathcal{C}$ vanishing
   takes the iteration step
   \begin{equation}\label{GPPA}
      z^{k+1}=z^k+\lambda_k(\mathcal{J}_{\gamma\mathcal{F}}z^k-z^k)\ \ {\rm for\ some}\ \gamma>0.
   \end{equation}
   When $\lambda_k\equiv 1$, the iteration \eqref{GPPA} reduces to that of the classical PPA \cite{Mar70,Roc76}.
   The linear convergence rate of the classical PPA is first established in \cite{Roc76} under
   the Lipschitz continuity of $\mathcal{F}^{-1}$ at $0$.
   Later, Artacho et al. \cite{ADG07} and  Leventhal \cite{Leventhal09} derived the linear convergence
   rate of the classical PPA under the metric regularity of $\mathcal{F}$ at a point $(\overline{z},0)\in{\rm gph}\mathcal{F}$
   and the metric subregularity of $\mathcal{F}$ at a point $(\overline{z},0)\in{\rm gph}\mathcal{F}$, respectively.
   Besides, the latter is weaker than the metric regularity of $\mathcal{F}$ at a point
   $(\overline{z},0)\in{\rm gph}\mathcal{F}$ and the Lipschitz continuity of $\mathcal{F}^{-1}$ near $0$
   in the sense of \cite{Roc76}. We next establish the linear convergence rate of the generalized PPA
   under the same assumption as in \cite{Leventhal09}.
   \begin{theorem}\label{theorem-GPPA}
    Assume that the operator $\mathcal{F}$ is maximally monotone. Let $\{z^{k}\}$ be given by
    the generalized PPA with $\{\lambda_k\}\subseteq[0,2]$ satisfying $\sum_{k=0}^{\infty}\lambda_k(2-\lambda_k)=+\infty$.
    Then,
   \begin{description}
   \item[(a)] the sequence $\{z^k\}$ converges to a point $z^*\in\mathcal{F}^{-1}(0)$, and moreover, it holds that
              \begin{equation}\label{GPPA-ineq}
                \|z^{k+1}\!-z\|^2 \le \|z^{k}\!-z\|^2-\lambda_k(2-\lambda_k)\|\mathcal{J}_{\gamma\mathcal{F}}z^{k}\!-z^k\|^2
                \quad \forall z\in\mathcal{F}^{-1}(0)\ {\rm and}\ \forall k\in\mathbb{N}.
              \end{equation}
   \item[(b)] If in addition $\mathcal{F}$ is metrically subregular at $(z^*,0)\in{\rm gph}\,\mathcal{F}$ with constant $\kappa>0$,
               then there exists $\overline{k}\in\mathbb{N}$ such that
               \begin{equation*}
                {\rm dist}\big(z^{k+1},\mathcal{F}^{-1}(0)\big)
                \le\sqrt{1-\frac{\lambda_k(2-\lambda_k)\gamma^2}{(\gamma+\kappa)^2}}\,{\rm dist}\big(z^{k},\mathcal{F}^{-1}(0)\big)
                \ \ {\rm for}\ k\ge \overline{k}.
               \end{equation*}
   \end{description}
  \end{theorem}
  \begin{proof}
  (a) Since $\mathcal{F}$ is maximally monotone, $\mathcal{J}_{\gamma\mathcal{F}}$
  is firmly nonexpansive (see \cite[Theorem 12.12]{RW98}), and so is $({1}/{2})$-averaged by
  \cite[Remark 4.24]{BC11}. The result follows by Lemma \ref{main-lemma2}.

  \medskip
  \noindent
  (b) Let $x^{k} = \mathcal{J}_{\gamma\mathcal{F}}z^{k}$.
  Notice that $z^*=\mathcal{J}_{\gamma\mathcal{F}}z^*$.
  We have $\|x^k-z^*\|\le\|z^k-z^*\|$ by the nonexpansiveness of $\mathcal{J}_{\gamma\mathcal{F}}$,
  which by part (a) implies that $x^k\to z^*$.
  Since $\mathcal{F}$ is metrically subregular at $(z^*,0)\in{\rm gph}\,\mathcal{F}$ with constant $\kappa>0$,
  there exists $\overline{k}\in\mathbb{N}$ such that
   \[
     {\rm dist}(x^{k},\mathcal{F}^{-1}(0)) \le\kappa{\rm dist}(0,\mathcal{F}(x^{k}))
   \le \frac{\kappa}{\gamma}\|z^{k}-x^{k}\|\quad\ {\rm for\ all}\ k\ge\overline{k},
   \]
   where the last inequality is due to
   \(
     \gamma^{-1}(z^k-x^{k})\in \mathcal{F}x^{k}
   \)
  implied by $x^{k} = \mathcal{J}_{\gamma\mathcal{F}}z^{k}$. Then,
  \[
   {\rm dist}(z^{k},\mathcal{F}^{-1}(0))\leq{\rm dist}(x^{k},\mathcal{F}^{-1}(0))+\|z^{k}-x^{k}\|
   \leq  (1+\frac{\kappa}{\gamma})\|z^{k}-x^{k}\|\ \ {\rm for}\ k\ge \overline{k}.
  \]
  Combining the last inequality with inequality \eqref{GPPA-ineq} yields that for all $k\ge\overline{k}$,
  \begin{align*}
  {\rm dist}(z^{k+1},\mathcal{F}^{-1}(0))^2
  &\leq \|z^{k+1}-\Pi_{\mathcal{F}^{-1}(0)}(z^{k})\|^2\nonumber\\
  &\leq \|z^{k}-\Pi_{\mathcal{F}^{-1}(0)}(z^{k})\|^2-\lambda_k(2-\lambda_k)\|z^{k}-x^k\|^2\nonumber\\
  &={\rm dist}(z^{k},\mathcal{F}^{-1}(0))^2 -\lambda_k(2-\lambda_k)\|z^{k}-x^k\|^2\nonumber\\
  &\leq \left[1-\frac{\lambda_k(2-\lambda_k)\gamma^2}{(\gamma+\kappa)^2}\right]{\rm dist}(z^{k},\mathcal{F}^{-1}(0))^2.
  \end{align*}
  where $\Pi_{\mathcal{F}^{-1}(0)}(\cdot)$ is the projection operator onto $\mathcal{F}^{-1}(0)$.
  The proof is completed.
  \end{proof}

  It is worthwhile to point out that Corman and Yuan \cite{CYuan14} derived
  the linear convergence rate of the generalized PPA under the strong monotonicity of $\mathcal{F}$,
  which is more stringent than the metric subregularity of
  $\mathcal{F}$ at $(\overline{x},\overline{y})\!\in{\rm gph}\mathcal{F}$.
  We also observe that  Liang et al. \cite{LFP15}
  establish the similar local linear convergence rate under the metric subregularity of
  $\big(\mathcal{I}-\mathcal{J}_{\gamma\mathcal{F}}\big)$ at $(z^*,0)\in {\rm gph}\,(\mathcal{I}-\mathcal{J}_{\gamma\mathcal{F}})$.
  By the Lemma \ref{relation-AB} in subsection \ref{relation},
  we will show that this condition is equivalent to the metric subregularity of
  $\mathcal{F}$ at $(z^*,0)\in {\rm gph}\mathcal{F}$. Very recently, Tao and Yuan \cite{TaoY16} also
  established the linear convergence rate of the generalized PPA under the Lipschitz continuity of $\mathcal{F}^{-1}$ 
  near 0 in the sense of \cite{Roc76}, which is stronger than the metric subregularity of 
  $\mathcal{F}$ at $(z,0)\in {\rm gph}\,\mathcal{F}$. 
 \subsection{Over-relaxed forward-backward splitting algorithm}\label{subsec3.2}

  The over-relaxed FBS algorithm for \eqref{composite-inclusion}
  with $\mathcal{B}$ vanishing takes the iteration steps:
  \begin{equation}\label{RFBS}
   z^{k+1}=z^k+\lambda_k(\mathcal{J}_{\gamma \mathcal{A}}\big(\mathcal{I}-\gamma\mathcal{C})z^k-z^k\big),
  \end{equation}
  which takes the form of equation \eqref{KM} with
  $\mathcal{T}:=\mathcal{J}_{\gamma\mathcal{A}}(\mathcal{I}-\gamma\mathcal{C})$ where $\gamma>0$ is the stepsize.
  When $\lambda_k\equiv 1$, the iteration \eqref{RFBS} reduces to the FBS algorithm studied in
  \cite{Gabay83,Passty79,CR06}.

  \medskip

  For the sequence $\{z^k\}$ generated by \eqref{RFBS},
  we have the following linear convergence result under the metric subregularity of $\mathcal{F}:=\mathcal{A}+\mathcal{C}$ at a point
  $(z^*,0)\in{\rm gph}\,\mathcal{F}$.
  \begin{theorem}\label{convergence-RFBS}
   Let $\mathcal{B}$ be $\beta$-cocoercive and  $\{z^k\}$ be the sequence generated by \eqref{RFBS}
   with $\gamma\!\in\!(0,2\beta)$ and $\lambda_k\!\in\![0,\delta]$
   such that $\sum_{k=0}^{\infty}\lambda_k(\delta\!-\!\lambda_k)\!=\!+\infty$,
   where $\delta\!=\!\min(1,\frac{\beta}{\gamma})\!+\!\frac{1}{2}$. Then,
   \begin{description}
    \item[(a)] the sequences $\{z^k\}$ and $\{\mathcal{T}z^k\}$ converge to a point $z^*\in{\rm Fix}\,\mathcal{T}=\mathcal{F}^{-1}(0)$,
               and
               \begin{equation}\label{RFBS-ineq1}
                \|z^{k+1}-z\|^2 \le \|z^{k}-z\|^2-\lambda_k(\delta-\lambda_k)\|\mathcal{T}z^{k}-z^k\|^2
                \quad\ \forall z\in \mathcal{F}^{-1}(0)\ {\rm and}\ \forall k\in\mathbb{N}.
               \end{equation}

    \item[(b)] If in addition $\mathcal{F}$ is metrically subregular at the point $(z^*,0)\!\in{\rm gph}\,\mathcal{F}$
               with constant $\kappa>0$, then there exists $\overline{k}\in\mathbb{N}$ such that for all $k\ge \overline{k}$,
               \begin{align*}
                {\rm dist}\big(z^{k+1},\mathcal{F}^{-1}(0)\big)\le\sqrt{1-\frac{\gamma^2\lambda_k(\delta-\lambda_k)}{(\gamma+\kappa)^2}}\,
                {\rm dist}\big(z^{k},\mathcal{F}^{-1}(0)\big).
               \end{align*}
   \end{description}
  \end{theorem}
  \begin{proof}
  (a) From the proof of \cite[Theorem 25.8]{BC11}, it follows that $\mathcal{T}$ is $({1}/{\delta})$-averaged.
  Thus, the result of part (a) follows directly from Lemma \ref{main-lemma2}.

  \medskip
  \noindent
  (b) Let $x^{k}=\mathcal{T}z^{k}$. From the definition of $\mathcal{T}$ and the single-valuedness of $\mathcal{C}$,
  we have
  \(
    z^k-\gamma\mathcal{C}z^k \in x^k+\gamma\mathcal{A}x^k.
  \)
  Hence,
  \(
    \gamma^{-1}(z^k-x^k)+\mathcal{C}x^{k}-\mathcal{C}z^{k}\in \mathcal{A}x^k+\mathcal{C}x^{k}=\mathcal{F}x^k.
  \)
  In addition, from part (a) it follows that $x^k\to z^*$. Now by the metric subregularity of $\mathcal{F}$
  at $(z^*,0) \in {\rm gph}\,\mathcal{F}$, there exists $\overline{k}\in\mathbb{N}$ such that
  for $k\ge\overline{k}$,
  \begin{align*}
  {\rm dist}(x^{k},\mathcal{F}^{-1}(0))
  &\le \kappa{\rm dist}(0,\mathcal{F}(x^k))
  \le \kappa\|\gamma^{-1}(z^k\!-\!x^{k})+\mathcal{C}x^{k}-\mathcal{C}z^{k}\|\nonumber\\
  &\le \kappa\sqrt{\gamma^{-2}\|z^k\!-\!x^{k}\|^2+\big(1-\frac{2\beta}{\gamma}\big)\|\mathcal{C}x^{k}-\mathcal{C}z^{k}\|^2}
  \le (\kappa/\gamma)\|x^k-z^{k}\|
  \end{align*}
  where the third inequality is using the cocoercivity of $\mathcal{C}$,
  and the last one is due to $\gamma\in(0,2\beta)$.
  The last inequality immediately implies that for $k\ge\overline{k}$,
  \begin{equation}\label{temp-equa411}
  {\rm dist}(z^{k},\mathcal{F}^{-1}(0))
   \le{\rm dist}(x^{k},\mathcal{F}^{-1}(0))+\|z^k-x^k\|\le\big(1+\kappa/\gamma\big)\big\|x^k-z^{k}\big\|.
  \end{equation}
  Combining inequality \eqref{temp-equa411} with inequality (\ref{RFBS-ineq1}), we obtain that
  for all $k\ge \overline{k}$,
  \begin{align*}
  {\rm dist}(z^{k+1},\mathcal{F}^{-1}(0))^2
  &\leq \big\|z^{k+1}-\Pi_{\mathcal{F}^{-1}(0)}(z^{k})\big\|^2\nonumber\\
  &\leq \big\|z^{k}-\Pi_{\mathcal{F}^{-1}(0)}(z^{k})\big\|^2-\lambda_k(\delta-\lambda_k)\|z^{k}-x^k\|^2\nonumber\\
  &\le \left[1-\frac{\gamma^2\lambda_k(\delta-\lambda_k)}{(\gamma+\kappa)^2}\right]
     {\rm dist}(z^{k},\mathcal{F}^{-1}(0))^2.
  \end{align*}
  This implies the desired result of part (b). The proof is then completed.
  \end{proof}

  By Theorem \ref{convergence-RFBS}, one may see that the linear convergence rate  coefficient is smallest
  when $\lambda_k=\frac{1}{2}\delta$. Recall that Chen and Rockafellar \cite{CR06} derived the linear convergence of
  the FBS algorithm under the strong monotonicity of $\mathcal{F}:=\mathcal{A}+\mathcal{C}$, which implies the single-valuedness
  and Lipschitz continuity of $\mathcal{F}^{-1}$, and then the metric subregularity of
  $\mathcal{F}$ at $(z^*,0)\in{\rm gph}\,\mathcal{F}$. Notice that
  the linear convergence is also derived in the work of Liang et al. \cite{LFP15} under the metric
  suregular of $\big(\mathcal{I}\!-\!\mathcal{T}\big)$
  at the point $(z^*,0)\in {\rm gph}\big(\mathcal{I}\!-\!\mathcal{T}\big)$.
  In the following lemma \ref{relation-AB} in subsection \ref{relation}, we show that the
  metric subregularity of $\big(\mathcal{I}\!-\!\mathcal{T}\big)$ is equivalent
  to the one of $\mathcal{F}$ at $(z^*,0)\in {\rm gph}(\mathcal{F})$.

 \subsection{Generalized Douglas-Rachford splitting algorithm}\label{sec3.3}

 Given $\gamma\!>\!0$, the generalized DRS method for \eqref{composite-inclusion} with $\mathcal{C}$ vanishing
 takes the iterations:
 \begin{numcases}{}\label{GDRS}
   z^{k} = \mathcal{J}_{\gamma\mathcal{B}}x^{k},\nonumber\\
   y^{k} = \mathcal{J}_{\gamma\mathcal{A}}(2z^{k}-x^{k}),\\
   x^{k+1} = x^{k}+\lambda_k(y^{k}-z^{k}),\nonumber
  \end{numcases}
 which can be rewritten as the form of \eqref{KM} with
 $\mathcal{T}:=\frac{1}{2}\big((2\mathcal{J}_{\gamma\mathcal{A}}-\mathcal{I})(2\mathcal{J}_{\gamma\mathcal{B}}-\mathcal{I})+\mathcal{I}\big)$, i.e.,
 \begin{equation}\label{GDRS1}
  x^{k+1} = x^{k}+\lambda_k(\mathcal{T}x^{k}-x^{k}).
 \end{equation}
 When $\lambda_k\equiv 1$, equation \eqref{GDRS} gives the DRS method \cite{LM79,DR56},
 and when $\lambda_k\equiv 2$ it gives the PRS method \cite{LM79,PR54}.
 Before stating the linear convergence rate of the generalized DRS method,
 we establish the relationship between the set $\mathcal{F}^{-1}(0)$ and
 the set ${\rm Fix}\,\mathcal{T}$.
 \begin{lemma}\label{zero-point1}
  The set $\mathcal{F}^{-1}(0)$ has the following relations with the fixed-point set ${\rm Fix}\,\mathcal{T}$:
  \begin{description}
    \item [(a)] $\mathcal{F}^{-1}(0)\!=\!\mathcal{J}_{\gamma\mathcal{B}}({\rm Fix}\,\mathcal{T})$. If $\mathcal{B}$ is single-valued,
                ${\rm Fix}\mathcal{T}\!=\!(\mathcal{I}+\gamma\mathcal{B})(\mathcal{F}^{-1}(0))$.
    \item [(b)] ${\rm Fix}\mathcal{T}\subseteq (\mathcal{I}\!-\!\gamma\!\mathcal{A})(\mathcal{F}^{-1}(0))$.
                 If $\mathcal{A}$ is single-valued, then ${\rm Fix}\mathcal{T}=(\mathcal{I}\!-\!\gamma\mathcal{A})(\mathcal{F}^{-1}(0))$.
  \end{description}
 \end{lemma}

  The proof of the above lemma is given in the Appendix.
  Next, we show the DRS method converges linearly
  under the metric subregularity of $\mathcal{F}$ at
  $(z^*,0)\!\in\!{\rm gph}\mathcal{F}$.
 \begin{theorem}\label{GDRS-convergence}
  Let $\{x^k\},\{y^k\}$ and $\{z^k\}$ be generated by the generalized DRS method with
  $\{\lambda_k\}\subseteq[0,2]$ such that $\sum_{k=0}^{\infty}\lambda_k(2-\lambda_k)=+\infty$.
  Then, the following statements hold.
  \begin{description}
   \item[(a)] $\{y^k\}$ and $\{z^k\}$ converge to $z^*\in\mathcal{F}^{-1}(0)$, and
               $\{x^k\}$ converges to $x^*\in{\rm Fix}\,\mathcal{T}$ and
               \begin{equation}\label{GDRS-ineq1}
                \|x^{k+1}-x\|^2 \le \|x^{k}-x\|^2-\lambda_k(2-\!\lambda_k)\|\mathcal{T}x^{k}-x^k\|^2
                \ \quad\forall x\in {\rm Fix}\,\mathcal{T}\ {\rm and}\ \forall k\in\mathbb{N}.
               \end{equation}

   \item[(b)] If $\mathcal{A}$ is single-valued and Lipschitz continuous with modulus $\frac{1}{\beta}$, and
               $\mathcal{F}$ is metrically subregular at $(z^*,0)\in{\rm gph}\,\mathcal{F}$ with constant $\kappa>0$,
               then there exists $\overline{k}\in\mathbb{N}$ such that
               \begin{equation*}
                 {\rm dist}(x^{k+1},{\rm Fix}\,\mathcal{T})
                 \!\le\!\sqrt{1\!-\!\frac{\lambda_k(2\!-\!\lambda_k)}
                        {\big[2\!+\!\sqrt{1\!+\!\gamma^2\beta^{-2}}\big(1\!+\!\kappa(\gamma^{-1}\!+\!\beta^{-1})\big)\big]^2}}
                 \,{\rm dist}(x^{k},{\rm Fix}\,\mathcal{T}),\,\forall k\!\ge\!\overline{k}.
                \end{equation*}

  \item[(c)] If $\mathcal{B}$ is single-valued and Lipschitz continuous with modulus $\frac{1}{\beta}$, and
               $\mathcal{F}$ is metric subregular at $(z^*,0)\in{\rm gph}\,\mathcal{F}$ with constant $\kappa>0$,
               then there exists $\overline{k}\in\mathbb{N}$ such that
               \begin{equation*}
                {\rm dist}(x^{k+1},{\rm Fix}\,\mathcal{T})\le\!\sqrt{1-\!\frac{\beta^2\lambda_k(2-\lambda_k)}{(\gamma\!+\!\beta)^2\big(1\!+\!\kappa\sqrt{\gamma^{-2}\!+\!\beta^{-2}}\big)^2}}
                \,{\rm dist}(x^{k},{\rm Fix}\,\mathcal{T}),\ \forall k\ge\overline{k}.
               \end{equation*}
  \end{description}
 \end{theorem}
 \begin{proof}
  (a) It is easy to check that $\mathcal{T}$ is nonexpansive.
  The result follows directly from Lemma \ref{main-lemma1} and
  the first equality of Lemma \ref{zero-point1}(a).

  \medskip
  \noindent
  (b) By the iteration steps \eqref{GDRS}, it follows that
  $\gamma^{-1}(x^k\!-\!z^k)\in\mathcal{B}z^{k}$, $\gamma^{-1}(z^{k}\!-\!x^{k}\!+\!z^{k}\!-\!y^{k})\!=\!\mathcal{A}y^{k}$
  and $\gamma^{-1}(z^{k}\!-\!y^{k})\in \mathcal{A}y^{k}\!+\!\mathcal{B}z^{k}$.
  This, along with the single-valuedness of $\mathcal{A}$, means that
  \[
  \gamma^{-1}(z^{k}-y^{k})+ \mathcal{A}z^{k}-\mathcal{A}y^{k}\in \mathcal{A}z^{k}+\mathcal{B}z^{k}=\mathcal{F}z^k.
  \]
  By part (a), using the metric subregularity of $\mathcal{F}$ at the point $(z^*,0)$ and
  the Lipschitz continuity of $\mathcal{A}$, it follows that there exists $\overline{k}\in\mathbb{N}$ such that for all $k\ge\overline{k}$,
  \begin{align*}
  {\rm dist}(z^{k},\mathcal{F}^{-1}(0))\le \kappa{\rm dist}(0,\mathcal{F}z^k)
  &\leq \kappa \big\|\gamma^{-1}(z^{k}-y^{k})+\mathcal{A}z^{k}-\mathcal{A}y^{k}\big\|\nonumber\\
  &\leq \kappa(\gamma^{-1}\!+\!\beta^{-1})\|z^{k}-y^{k}\|,
  \end{align*}
  which further implies that
 \begin{equation}\label{GDRS-ineq422}
  {\rm dist}(y^{k},\mathcal{F}^{-1}(0))
  \le {\rm dist}(z^{k},\mathcal{F}^{-1}(0)) +\|z^{k}-y^{k}\|
  \le \big[1+\kappa(\gamma^{-1}\!+\!\beta^{-1})\big]\|y^{k}-z^{k}\|.
 \end{equation}
  Let $\overline{z}^k=\Pi_{\mathcal{F}^{-1}(0)}(y^k)-\gamma\mathcal{A}(\Pi_{\mathcal{F}^{-1}(0)}(y^k))$.
  From Lemma \ref{zero-point1}(b), $\overline{z}^k\in{\rm Fix}\mathcal{T}$.
  In addition, notice that $2z^k-x^k=y^k+\gamma\mathcal{A}y^k$ by the second equality in \eqref{GDRS}.
  Thus, for all $k\ge \overline{k}$,
  \begin{align*}
   {\rm dist}(x^{k},{\rm Fix}\,\mathcal{T})
   &\le\|x^k-\Pi_{\mathcal{F}^{-1}(0)}(y^k)+\gamma\mathcal{A}(\Pi_{\mathcal{F}^{-1}(0)}(y^k))\|\\
   &= \|2z^k-y^k-\gamma\mathcal{A}y^k-\Pi_{\mathcal{F}^{-1}(0)}(y^k)+\gamma\mathcal{A}(\Pi_{\mathcal{F}^{-1}(0)}(y^k))\|\\
   &\le \|2z^k-2y^k\|+\|y^k-\Pi_{\mathcal{F}^{-1}(0)}(y^k)-\gamma\mathcal{A}y^k+\gamma\mathcal{A}(\Pi_{\mathcal{F}^{-1}(0)}(y^k))\| \\
   &\le 2\|z^k-y^k\|+\sqrt{1+\gamma^2\beta^{-2}}{\rm dist}(y^k,\mathcal{F}^{-1}(0))\\
   &\le \big[2+\!\sqrt{1+\gamma^2\beta^{-2}}(1+\kappa(\gamma^{-1}\!+\!\beta^{-1}))\big]\|z^k-y^k\|,
  \end{align*}
  where the third inequality is using the Lipschitz continuity and monotonicity of $\mathcal{A}$,
  and the last one is due to \eqref{GDRS-ineq422}. In addition, from equation \eqref{GDRS1} and
  the last equality of \eqref{GDRS}, we have
  \(
    \mathcal{T}x^{k}-x^k=\frac{(x^{k+1}-x^k)}{\lambda_k}=(y^k-z^k),
  \)
  which together with \eqref{GDRS-ineq1} implies that
  \begin{align}\label{GDRS-ineq423}
  {\rm dist}(x^{k+1},{\rm Fix}\,\mathcal{T})^2
  &\leq \big\|x^{k+1}-\Pi_{{\rm Fix}\,\mathcal{T}}(x^{k})\big\|^2\nonumber\\
  &\leq \big\|x^{k}-\Pi_{{\rm Fix}\,\mathcal{T}}(x^{k})\big\|^2-\lambda_k(2-\lambda_k)\|\mathcal{T}x^{k}-x^k\|^2\nonumber\\
  &= {\rm dist}(x^{k},{\rm Fix}\,\mathcal{T})^2-\lambda_k(2-\lambda_k)\|\mathcal{T}x^{k}-x^k\|^2\nonumber\\
  &= {\rm dist}(x^{k},{\rm Fix}\,\mathcal{T})^2-\lambda_k(2-\lambda_k)\|y^k-z^k\|^2.
  \end{align}
  The desired result of part (b) follows directly from the last two inequalities.

  \medskip
  \noindent
  (c) Notice that $\mathcal{B}$ is single-valued and Lipschitzian. Thanks to part(a) and (b), we have
 \[
  \gamma^{-1}(z^{k}-y^{k})+ \mathcal{B}y^k-\mathcal{B}z^k\in \mathcal{A}y^{k}+\mathcal{B}y^{k}=\mathcal{F}y^k
 \]
 From the metric subregularity of $\mathcal{F}$ at the point $(z^*,0)$,
 there exists $k\ge\overline{k}$ such that
  \begin{align*}
  {\rm dist}(y^{k},\mathcal{F}^{-1}(0))&\le \kappa{\rm dist}(0,\mathcal{F}y^k)
  \le \kappa \big\|\gamma^{-1}(z^{k}-y^{k})+ \mathcal{B}y^k-\mathcal{B}z^k\big\|\nonumber\\
  &\le \kappa\sqrt{\gamma^{-2}\!+\!\beta^{-2}}\|z^{k}-y^{k}\|\quad{\rm for}\ k\ge\overline{k},
  \end{align*}
  where the last inequality is using the monotonicity of $\mathcal{B}$.
  Consequently, for all $k\ge\overline{k}$,
 \[
  {\rm dist}(z^{k},\mathcal{F}^{-1}(0))
  \le {\rm dist}(y^{k},\mathcal{F}^{-1}(0)) +\|z^{k}-y^{k}\|
  \le \big[1+\kappa\sqrt{\gamma^{-2}\!+\!\beta^{-2}}\big]\|y^{k}-z^{k}\|.
 \]
  Let $\overline{z}^k=\mathcal{B}(\Pi_{\mathcal{F}^{-1}(0)}(z^k))$. By Lemma \ref{zero-point1}(a),
  clearly, $\Pi_{\mathcal{F}^{-1}(0)}(z^k)+\gamma\overline{z}^k\in {\rm Fix}\,\mathcal{T}$.
  Moreover, using the Lipschitzian property of $\mathcal{B}$, we have
  \(
    \|\overline{z}^k-\mathcal{B}z^k\|\le \beta^{-1}\|\Pi_{\mathcal{F}^{-1}(0)}(z^k)-z^k\|.
  \)
  In addition, according to  $x^k=z^k+\gamma\mathcal{B}z^k\in \mathcal{J}_{\gamma\mathcal{B}}^{-1}(z^k)$, we further obtain
  \begin{align*}
    {\rm dist}(x^{k},{\rm Fix}\,\mathcal{T})
   &\le\|x^k-\Pi_{\mathcal{F}^{-1}(0)}(z^k)-\gamma\overline{z}^k\|\\
   &\le \|z^k-\Pi_{\mathcal{F}^{-1}(0)}(z^k)\|+\gamma\|\widetilde{x}^k-\overline{z}^k\|\\
   &\le (1+\gamma\beta^{-1})\|z^k-\Pi_{\mathcal{F}^{-1}(0)}(z^k)\|\\
   &\le(1+\gamma\beta^{-1})\big[1+\kappa\sqrt{\gamma^{-2}\!+\!\beta^{-2}}\big]\|y^{k}-z^{k}\|
   \quad{\rm for}\ k\ge\overline{k}.
  \end{align*}
   Combining this inequality with \eqref{GDRS-ineq423} yields the desired result.
   The proof is completed.
  \end{proof}

%
  \begin{remark}
  Giselsson \cite{Giselsson15} and Davis et al. \cite{Davis152} recently derived the linear convergence rate
  of the generalized DRS method under the assumption that $\mathcal{A}$ is strongly monotone
  and $\mathcal{B}$ is $\beta$-Lipschitz continuous, which is stronger than the assumption of
  Theorem \ref{GDRS-convergence}(c).
  \end{remark}

  It is well known that the generalized DRS method is a generalized PPA associated with
  operator $\mathcal{S}_{\gamma,\mathcal{A},\mathcal{B}}$  in the sense that
  $\mathcal{T}\!=\!\big(\mathcal{I}\!+\!\mathcal{S}_{\gamma,\mathcal{A},\mathcal{B}}\big)^{-1}
  \!=\!\mathcal{J}_{ \mathcal{S}_{\gamma,\mathcal{A},\mathcal{B}}}$ by \cite[Theorem 5]{EB92}.
  That is, the sequence $\{x^k\}$ in \eqref{GDRS} can be generated by the following iteration step
  \begin{equation}\label{GDRS2}
    x^{k+1}=x^k+\lambda_k(\mathcal{J}_{ \mathcal{S}_{\gamma,\mathcal{A},\mathcal{B}}}x^k-x^k).
  \end{equation}
  By Theorem \ref{theorem-GPPA}, we also have the linear convergence rate of the generalized DRS method
  under the metric subregularity of $\mathcal{S}_{\gamma,\mathcal{A},\mathcal{B}}$
  at $(\overline{x},0)\in{\rm gph}\mathcal{S}_{\gamma,\mathcal{A},\mathcal{B}}$, stated as follows.
 \begin{theorem}\label{GDRS1-convergence}
   Let $\{x^{k}\}$ be the sequence generated by equation \eqref{GDRS2} with $\{\lambda_k\}\subseteq[0,2]$
   and $\sum_{k=0}^{\infty}\lambda_k(2-\!\lambda_k)=+\infty$. Then, the following statements hold.
   \begin{description}
    \item[(a)] $\{z^k\}$ and $\{x^k\}$ converge to $z^*\in\mathcal{F}^{-1}(0)$ and
               $x^*\in\mathcal{S}_{\gamma,\mathcal{A},\mathcal{B}}^{-1}(0)$, respectively, and
               \begin{equation*}
                \|x^{k+1}-x\|^2 \le \|x^{k}-x\|^2-\lambda_k(2-\lambda_k)\|\mathcal{J}_{\mathcal{S}_{\gamma,\mathcal{A},\mathcal{B}}}x^{k}\!-x^k\|^2
                \quad \forall x\in\mathcal{S}_{\gamma,\mathcal{A},\mathcal{B}}^{-1}(0)\ {\rm and}\ \forall k\in\mathbb{N}.
               \end{equation*}

    \item[(b)] If $\mathcal{S}_{\gamma,\mathcal{A},\mathcal{B}}$ is metrically subregular at $(x^*,0)$
               with constant $\kappa>0$, then there exists $\overline{k}\in\mathbb{N}$ such that for all $k\ge\overline{k}$,
               \begin{equation}\label{DRS-PPA}
                {\rm dist}\big(x^{k+1},\mathcal{S}_{\gamma,\mathcal{A},\mathcal{B}}^{-1}(0)\big)
                \le\sqrt{1-\frac{\lambda_k(2-\lambda_k)}{(1+\kappa)^2}}\,{\rm dist}\big(x^{k},\mathcal{S}_{\gamma,\mathcal{A},\mathcal{B}}^{-1}(0)\big).
               \end{equation}
   \end{description}
  \end{theorem}

  Corman and Yuan \cite{CYuan14} derived the linear convergence rate of the generalized
  DRS method under the assumption that $\mathcal{S}_{\gamma,\mathcal{A},\mathcal{B}}$ is strongly
  monotone (implied by the strong monotonicity of $\mathcal{F}$ and one of $\mathcal{A}$
  and $\mathcal{B}$ is firmly nonexpansive), which is stronger than the metric subregularity
  of $\mathcal{S}_{\gamma,\mathcal{A},\mathcal{B}}$ by Proposition \ref{prop4.1} in subsection \ref{relation}.
  More recently,  Liang et al. \cite{LFP15} establish its local linear convergence rate like (\ref{DRS-PPA})
   under the metric subregularity of $\big(\mathcal{I}-\mathcal{T}\big)$ at a point
  $(z^*,0)\in {\rm gph}\,(\mathcal{I}-\mathcal{T})$
  which is equivalent to the metric subregular of $\mathcal{S}_{\gamma,\mathcal{A},\mathcal{B}}$ at
  $(z^*,0)\in {\rm gph}\,\mathcal{S}_{\gamma,\mathcal{A},\mathcal{B}}$ according to lemma \ref{relation-AB}
  in the subsection \ref{relation}.

  \medskip

  Although, the linear convergence of the generalized DRS algorithm can be deriveed
  under the metric subregularity of $\mathcal{S}_{\gamma,\mathcal{A},\mathcal{B}}$
  or $\big(\mathcal{I}-\mathcal{T}\big)$ at a point of its graph, this regular condition
  may be too difficult to be certified since that $\mathcal{S}_{\gamma,\mathcal{A},\mathcal{B}}$
   is highly compound of $\mathcal{A}$ and $\mathcal{B}$.
  On the contrast, the metric subregularity of $\mathcal{F}:=\mathcal{A}+\mathcal{B}$ at the point
   $(z^*,0)\in {\rm gph}\,\mathcal{F}$ may be slightly
  easier to check due to its simple formulation. In the last subsection, we will give some sufficient
  conditions to ensure the metric subregularity of $\mathcal{F}=\mathcal{A}+\mathcal{B}$.
  \subsection{Davis' three-operator splitting algorithm}\label{subsec3.4}

  Davis's splitting method \cite{Davis15} for the inclusion problem \eqref{composite-inclusion}
  takes the following iterations
  \begin{numcases}{}\label{Davis-method}
   z^{k} = \mathcal{J}_{\gamma\mathcal{B}}(x^{k}),\nonumber\\
   y^{k} = \mathcal{J}_{\gamma\mathcal{A}}(2z^{k}-x^{k}-\gamma\mathcal{C}z^k),\\
   x^{k+1} = x^{k}+\lambda_k(y^{k}-z^{k}).\nonumber
  \end{numcases}
  Let
  \(
  \mathcal{T}:=\mathcal{I}-\mathcal{J}_{\gamma\mathcal{B}}+\mathcal{J}_{\gamma\mathcal{A}}
               \circ(2\mathcal{J}_{\gamma\mathcal{B}}-\mathcal{I}-\gamma\mathcal{C}\circ\mathcal{J}_{\gamma\mathcal{B}}).
  \)
  Then, with this operator, the iterations in equation \eqref{Davis-method} can be compactly
  written as
  \(
    x^{k+1} = x^{k}+\lambda_k(\mathcal{T}x^{k}-x^{k}).
  \)

  \medskip

 The following lemma present the relation between the solution set $\mathcal{F}^{-1}(0)$
 and the fixed-point set ${\rm Fix}\,\mathcal{T}$. For the sake of coherence, its proof is given in the appendix.
 \begin{lemma}\label{zero-point2}
  The set $\mathcal{F}^{-1}(0)$ has the following relations with the fixed-point set ${\rm Fix}\,\mathcal{T}$:
   \begin{description}
    \item [(a)] $\mathcal{F}^{-1}(0)=\mathcal{J}_{\gamma\mathcal{B}}({\rm Fix}\,\mathcal{T})$. If $\mathcal{B}$ is single-valued,
                ${\rm Fix}\mathcal{T}=(\mathcal{I}+\gamma\mathcal{B})(\mathcal{F}^{-1}(0))$.
    \item [(b)] ${\rm Fix}\mathcal{T}\!\subseteq\!(\mathcal{I}\!-\gamma(\mathcal{A}+\mathcal{C}))(\mathcal{F}^{-1}(0))$.
                 If $\mathcal{A}$ is single-valued, ${\rm Fix}\mathcal{T}\!=\!(\mathcal{I}\!-\gamma(\mathcal{A}+\mathcal{C}))(\mathcal{F}^{-1}(0))$.
  \end{description}
 \end{lemma}
 \begin{theorem}\label{Davis-convergence}
   Let $\{x^k\}$ be the sequence generated by \eqref{Davis-method} with
   $\{\lambda_k\}\subseteq[0,\frac{4\vartheta-\gamma}{2\vartheta}]$ such that
   $\sum_{k=0}^{\infty}\lambda_k(\frac{4\vartheta-\gamma}{2\vartheta}-\lambda_k)=+\infty$
   for $\gamma\in(0,2\vartheta)$. Then, the following statements hold.
   \begin{description}
    \item[(a)] $\{y^k\}$ and $\{z^k\}$ converge to $z^*\in\mathcal{F}^{-1}(0)$, and
               $\{x^k\}$ converges to $x^*\in{\rm Fix}\,\mathcal{T}$, and
               \begin{equation*}
                \|x^{k+1}\!-\!x\|^2 \le \|x^{k}\!-\!x\|^2-\!\lambda_k\big(\frac{4\vartheta\!-\!\gamma}{2\vartheta}-\!\lambda_k\big)\|\mathcal{T}x^{k}-x^k\|^2
                \quad\forall x\in {\rm Fix}\,\mathcal{T}\ {\rm and}\ \forall k\in\mathbb{N}.
               \end{equation*}

    \item[(b)] If $\mathcal{A}$ is single-valued and Lipschitz continuous with modulus $1/\beta$  and
               $\mathcal{F}$ is metric subregular with constant $\kappa>0$ at $(z^*,0)$,
               then there exists $\overline{k}\in\mathbb{N}$ such that
               \begin{equation}\label{Davis-ineq2}
                  {\rm dist}^2(x^{k+1},{\rm Fix}\,\mathcal{T})\le (1-\!\varrho)\,{\rm dist}^2(x^{k},{\rm Fix}\,\mathcal{T})
                  \quad\ {\rm for}\ k\ge\overline{k}
              \end{equation}
               with
               \[
                  \varrho=\frac{\lambda_k(4\vartheta-\gamma-2\theta\lambda_k)}{ 2\vartheta
                   \big[(2+\gamma\vartheta^{-1})+(\gamma\vartheta^{-1}\!+\!\sqrt{1+\gamma^2\beta^{-2}})(1+\kappa(\gamma^{-1}\!+\!\beta^{-1}))\big]^2}.
               \]

    \item[(c)]  If $\mathcal{B}$ is single-valued and Lipschitz continuous with modulus $1/\beta$ and
               $\mathcal{F}$ is metric subregular with constant $\kappa>0$ at $(z^*,0)$, then there exists
               $\overline{k}\in\mathbb{N}$ such that \eqref{Davis-ineq2} holds for all $k\ge \overline{k}$ with
                \[
                  \varrho=\frac{\lambda_k(4\vartheta-\gamma-2\theta\lambda_k)}{ 2\vartheta(1+\gamma\beta^{-1})
                   \Big[1+\kappa\big(\frac{1}{\beta}+\frac{1}{\gamma}\sqrt{1+\max\big(\frac{\gamma^2-2\gamma\vartheta}{\vartheta^2},0\big)}\big)\Big]^2}.
               \]

   \end{description}
 \end{theorem}
 \begin{proof}
  (a) By \cite[Proposition 3.1]{Davis15}, $\mathcal{T}$ is $\alpha$-averaged
  with $\alpha=\frac{2\vartheta}{4\vartheta-\gamma}$ for $\gamma\in(0,2\vartheta)$.
  By Lemma \ref{zero-point2}(a),
  $\mathcal{F}^{-1}(0)=\mathcal{J}_{\gamma\mathcal{B}}({\rm Fix}\,\mathcal{T})$.
  Thus, the result directly follows by Lemma \ref{main-lemma2}.

  \medskip
  \noindent
 (b) From the iteration step \eqref{Davis-method},
  \(
    \gamma^{-1}(x^{k}-z^{k})\in\mathcal{B}z^{k}
  \)
  and
  \(
    \gamma^{-1}(2z^{k}-x^{k}-\gamma\mathcal{C}z^{k}-y^{k})\in \mathcal{A}y^{k}.
  \)
  Hence, we have
 \(
    \gamma^{-1}(z^{k}-y^{k})\in \mathcal{A}y^{k}+\mathcal{B}z^{k}+\mathcal{C}z^{k},
 \)
 which further implies that
 \[
   \gamma^{-1}(z^k-y^k)+\mathcal{A}z^k-\mathcal{A}y^k\in\mathcal{A}z^{k}+\mathcal{B}z^{k}+\mathcal{C}z^{k}\in\mathcal{F}z^k.
 \]
 Since $\mathcal{F}$ is metrically subregular at $(z^*,0)$ with constant $\kappa$
 and $z^k\to z^*$ by part (a), there exists $\overline{k}\in\mathbb{N}$ such that for all $k\ge \overline{k}$,
 the latter inequalities hold
 \begin{align*}
  {\rm dist}(z^{k},\mathcal{F}^{-1}(0))
  \leq\! \kappa {\rm dist}(0,\mathcal{F}(z^{k}))
  \!\leq\! \kappa\|\gamma^{-1}(z^{k}\!-\!y^{k})\!+\!\mathcal{A}z^k\!-\!\mathcal{A}y^k\|
  \!\le\! \kappa(\gamma^{-1}\!+\!\beta^{-1})\|z^k\!-\!y^k\|
 \end{align*}
 where the last inequality is due to the Lipschitz continuity of $\mathcal{A}$. Thus, for all $k\ge \overline{k}$,
 \begin{equation}\label{Davis-ineq432}
  {\rm dist}(y^{k},\mathcal{F}^{-1}(0))
  \le  {\rm dist}(z^{k},\mathcal{F}^{-1}(0)) +\|z^{k}-y^{k}\|
  \le \Big[1+\kappa(\gamma^{-1}+\beta^{-1})\Big]\|z^k-y^k\|.
  \end{equation}
  Let $\overline{z}^k=\Pi_{\mathcal{F}^{-1}(0)}(y^k)-\gamma(\mathcal{A}+\mathcal{C})(\Pi_{\mathcal{F}^{-1}(0)}(y^k))$.
  From Lemma \ref{zero-point2}(b), $\overline{z}^k\in{\rm Fix}\mathcal{T}$. In addition, notice that
  $x^k=2z^k-y^k-\gamma\mathcal{A}y^k-\gamma\mathcal{C}z^k$ by \eqref{Davis-method}. Hence,
  for all $k\ge\overline{k}$,
  \begin{align*}
   {\rm dist}(x^{k},{\rm Fix}\,\mathcal{T})
   &\le\|x^k-\Pi_{\mathcal{F}^{-1}(0)}(y^k)+\gamma(\mathcal{A}+\mathcal{C})(\Pi_{\mathcal{F}^{-1}(0)}(y^k))\|\\
   &= \|2z^k-y^k-\gamma\mathcal{A}y^k-\gamma\mathcal{C}z^k-\Pi_{\mathcal{F}^{-1}(0)}(y^k)+\gamma(\mathcal{A}+\mathcal{C})(\Pi_{\mathcal{F}^{-1}(0)}(y^k))\|\\
   &\le \|2z^k-2y^k\|+\|y^k-\Pi_{\mathcal{F}^{-1}(0)}(y^k)-\gamma\mathcal{A}y^k+\gamma\mathcal{A}(\Pi_{\mathcal{F}^{-1}(0)}(y^k))\| \\
   &\quad +\|\gamma \mathcal{C}y^k-\gamma\mathcal{C}z^k+\gamma\mathcal{C}(\Pi_{\mathcal{F}^{-1}(0)}(y^k))-\gamma \mathcal{C}y^k\|\\
   &\le (2+\gamma\vartheta^{-1})\|z^k-y^k\|+(\gamma\vartheta^{-1}+\sqrt{1+\gamma^2\beta^{-2}}){\rm dist}(y^k,\mathcal{F}^{-1}(0))\\
   &\le \big[(2+\gamma\vartheta^{-1})+(\gamma\vartheta^{-1}\!+\!\sqrt{1+\gamma^2\beta^{-2}})(1+\kappa(\gamma^{-1}\!+\!\beta^{-1}))\big]\|z^k-y^k\|,
  \end{align*}
  where the third inequality is using the monotonicity of $\mathcal{A}$ and
  the last one is due to \eqref{Davis-ineq432}. In addition, using part (a) and
  the same arguments as for Theorem \ref{GDRS-convergence}(b) yields
  \begin{equation}\label{Davis-ineq433}
  {\rm dist}(x^{k+1},{\rm Fix}\,\mathcal{T})^2
  \le {\rm dist}(x^{k},{\rm Fix}\,\mathcal{T})^2-\lambda_k\big(\frac{4\vartheta-\gamma}{2\vartheta}-\!\lambda_k\big)\|y^k-z^k\|^2.
  \end{equation}
  The desired result then follows from the last two inequalities.

 \medskip
 \noindent
 (c) By the proof of part (b), we have
 $\gamma^{-1}(z^{k}-y^{k}) \in \mathcal{A}y^{k}+\mathcal{B}z^{k}+\mathcal{C}z^{k}$.
 This along with the single-value $\mathcal{B}$ yield
 \(
   \gamma^{-1}(z^k-y^k)+(\mathcal{B}y^k-\mathcal{B}z^k)\in\mathcal{A}y^{k}+\mathcal{B}y^{k}+\mathcal{C}z^{k}.
 \)
 By the cocoercivity of $\mathcal{C}$, we obtain that
 \(
  \gamma^{-1}(z^k-y^k)+(\eta^k-\xi^k)+\mathcal{C}y^{k}-\mathcal{C}z^{k}
  \in \mathcal{A}y^{k}+\mathcal{B}y^{k}+\mathcal{C}y^{k}=\mathcal{F}y^k.
 \)
 Since $\mathcal{F}$ is metrically subregular at $(z^*,0)$ with constant $\kappa$,
 from part (a) and the Lipschitzian of $\mathcal{B}$ and the cocoercivity of $\mathcal{C}$,
  it follows that there exists $\overline{k}$ such that for all $k\ge \overline{k}$,
 \begin{align*}
  {\rm dist}(y^{k},\mathcal{F}^{-1}(0))
  &\leq \kappa {\rm dist}(0,\mathcal{F}(y^{k}))
  \leq \kappa\|\gamma^{-1}(z^{k}-y^{k})+(\eta^k-\xi^k)+\mathcal{C}y^{k}-\mathcal{C}z^{k}\|\nonumber\\
  &\le \kappa\left[\|\eta^k-\xi^k\|+\|\gamma^{-1}(z^{k}-y^{k})+\mathcal{C}y^{k}-\mathcal{C}z^{k}\|\right]\nonumber\\
  &\leq \kappa\Big[\frac{1}{\beta}+\sqrt{\frac{1}{\gamma^2}+\max\big(\frac{\gamma-2\vartheta}{\gamma\vartheta^2},0\big)}\Big]\|z^k-y^k\|.
 \end{align*}
  Combing this inequality with  ${\rm dist}(z^{k},\mathcal{F}^{-1}(0))\!\leq\! {\rm dist}(y^{k},\mathcal{F}^{-1}(0))\!+\!\|z^k\!-\!y^k\|$ implies that
 \begin{align*}
  {\rm dist}(z^{k},\mathcal{F}^{-1}(0))
  \le \left[1+\kappa\Big(\frac{1}{\beta}+\frac{1}{\gamma}\sqrt{1+\max\big(\frac{\gamma^2-2\gamma\vartheta}{\vartheta^2},0\big)}\Big)\right]\|z^k-y^k\|,\,
  \forall k\ge \overline{k}
  \end{align*}
  From ${\rm Fix}\,\mathcal{T}=\mathcal{J}_{\gamma\mathcal{B}}^{-1}\mathcal{F}^{-1}(0)$,
  it follows that
  \(
    \Pi_{\mathcal{F}^{-1}(0)}(z^k)+\gamma\mathcal{B}(\Pi_{\mathcal{F}^{-1}(0)}(z^k))\in {\rm Fix}\,\mathcal{T}.
  \)
  Moreover, using the
  Lipschitzian of $\mathcal{B}$ yields
  \(
    \|\mathcal{B}(\Pi_{\mathcal{F}^{-1}(0)}(z^k))-\mathcal{B}z^k\|\le \beta^{-1}\|\Pi_{\mathcal{F}^{-1}(0)}(z^k)-z^k\|.
  \)
  Combining this inequality with the facts that $x^k=z^k+\gamma\mathcal{B}z^k$ and $\Pi_{\mathcal{F}^{-1}(0)}(z^k)\to z^*$,
  and using the same arguments as for those of Theorem \ref{GDRS-convergence}(c) yield that for all $k\ge \overline{k}$,
  \[
    {\rm dist}(x^{k},{\rm Fix}\,\mathcal{T})
    \le(1+\gamma\beta^{-1})\left[1+\kappa\Big(\frac{1}{\beta}+\frac{1}{\gamma}\sqrt{1+\max\big(\frac{\gamma^2-2\gamma\vartheta}{\vartheta^2},0\big)}\Big)\right]\|z^k-y^k\|.
  \]
  Combining this inequality with \eqref{Davis-ineq433} yields the desired result. The proof is completed.
 \end{proof}

 Davis and Yin \cite{Davis15} derived the linear convergence rate of their algorithm
 under the condition that one of $\mathcal{A},\mathcal{B}$ and $\mathcal{C}$ is strongly monotone
 and one of $\mathcal{A}$ and $\mathcal{B}$ is single-valued and Lipschitz continuous, which is stronger than that of
 Theorem \ref{Davis-convergence} (b) and (c).

  \subsection{Sufficient conditions for the metric subregularity}\label{relation}
  In this section, we give some sufficient conditions to ensure the metric subregularity
  of the maximal monotone operator $\mathcal{F}:=\mathcal{A}+\mathcal{B}$
  under the condition that $\mathcal{B}$ is single valued and Lipshcitz
  continuous with modulus $\frac{1}{\beta}$.
  Write $\mathcal{R}(z)=z-\mathcal{J}_{\gamma\mathcal{A}}(\mathcal{I}-\gamma\mathcal{B})z$
  which is clearly a single valued mapping. In the following lemma, we give an equivalent
  characterization on the metrically subregularity of $\mathcal{F}$ at
  ${z}^*$ for $0\in \mathcal{F}(z^*)$.
  \begin{lemma}\label{relation-AB}
   Let $\mathcal{F}:=\mathcal{A}+\mathcal{B}$ where $\mathcal{A}$ is maximal monotone and
   $\mathcal{B}$ is single valued and Lipschitz continuous with modulus $\frac{1}{\beta}$. Then,
   the mapping $R$ is metrically subregular at ${z}^*$ for
   $0\in \mathcal{R}({z}^*)$ if and only if the operator $\mathcal{F}$ is metrically subregular
   at ${z}^*$ for $0\in \mathcal{F}({z}^*)$.
  \end{lemma}
  \begin{proof}
   By the metric subregularity of $R$ at $z^*$ for $0\in R(z^*)$,
  there exist a constant $\kappa'>0$ and a sufficiently small $\delta'>0$ such that
  \begin{equation}\label{Lemma-AB-eq1}
    {\rm dist}\big(z,\mathcal{R}^{-1}(0)\big) \leq \kappa'\|\mathcal{R}(z)\|,
     \quad\forall z\in \mathbb{B}(z^*,\delta')
  \end{equation}
  where $\mathbb{B}(z^*,\delta')$ denotes the closed ball in the space $\mathbb{H}$
  centered at $z^*$ with radius $\delta'$.
  Let $z$ be arbitrary point from $\mathbb{B}(z^*,\delta')$.
  Take $y\!\in\!\mathcal{F}(z)$ as the point such that ${\rm dist}(0,\mathcal{F}(z))\!=\!\|y\|$.
  Notice that $\mathcal{F}\!=\!\mathcal{A}+\mathcal{B}$ and $\mathcal{B}$ is single valued and lipschitz continuous,
  it is easy to get that $z+\gamma y-\gamma\mathcal{B}z\in (\mathcal{I}\!+\!\gamma\mathcal{A})z$ which
  in turn implies that $z = \mathcal{J}_{\gamma\mathcal{A}}\big(z+\gamma y-\gamma\mathcal{B}z\big)$.
  Together with the last equation and the metric subregularity (\ref{Lemma-AB-eq1}) of $\mathcal{R}$, we obtain
  that
  \begin{align*}
  {\rm dist}\big(z, \mathcal{F}^{-1}(0)\big)
  \!=\!{\rm dist}\big(z,\mathcal{R}^{-1}(0)\big)\!\leq\! \kappa\|\mathcal{R}(z)\|
  \!=\! \kappa'\|\mathcal{J}_{\gamma\mathcal{A}}\big(z+\gamma y-\gamma\mathcal{B}z\big)-\mathcal{J}_{\gamma\mathcal{A}}(z-\gamma\mathcal{B}z)\|.
  \end{align*}
  Notice that $\mathcal{J}_{\gamma\mathcal{A}}$ is nonexpansive. This along with the
  equality ${\rm dist}\big(0,\mathcal{F}(z)\big)=\|y\|$ yield
    \begin{align*}
  {\rm dist}\big(z, \mathcal{F}^{-1}(0)\big)
  \le \kappa'\gamma{\rm dist}\big(0,\mathcal{F}(z)\big),
  \end{align*}
  which shows that the operator $\mathcal{F}$ is metrically subregular at $z^*$ for $0\in \mathcal{F}(z^*)$.

  \medskip
  Conversely, suppose that $\mathcal{F}$ is metrically subregular at $z^*$ for $0\in \mathcal{F}(z^*)$.
  Then, there exist a constant $\kappa>0$ and a sufficiently small $\delta>0$ such that
  \begin{equation}\label{temp-equa0}
    {\rm dist}(z,\mathcal{F}^{-1}(0))\le \kappa{\rm dist}(0,\mathcal{F}(z))\quad\forall z\in\mathbb{B}(z^*,\delta).
  \end{equation}
  Take $\delta'\!=\!\frac{\delta}{(1+\gamma\beta^{-1})}$.
  Notice that the equation holds $\mathcal{J}_{\gamma\mathcal{A}}(\mathcal{I}\!-\!\gamma\mathcal{B})z^*=z^*$ since $z^*\in \mathcal{F}^{-1}(0)$.
  Then, $\mathcal{J}_{\gamma\mathcal{A}}(\mathcal{I}\!-\!\gamma\mathcal{B})z \in \mathbb{B}(z^*,\delta)$ due to
  $\|\mathcal{J}_{\gamma\mathcal{A}}(\mathcal{I}\!-\!\gamma\mathcal{B})z\!-\!z^*\|\leq (1\!+\!\gamma\beta^{-1})\|z-z^*\|\!\le\! \delta$.
  Combine $z\!-\!\gamma\mathcal{B}z\!-\!\mathcal{J}_{\gamma\mathcal{A}}(z\!-\!\gamma\mathcal{B}z)
              \!\in\!\gamma\mathcal{A}\big(\mathcal{J}_{\gamma\mathcal{A}}(z\!-\!\gamma\mathcal{B}z)\big)$
  and metric subregularity of $\mathcal{F}$ yield
  \begin{align*}
  &{\rm dist}\big(z,\mathcal{R}^{-1}(0)\big)
  ={\rm dist}\big(z, \mathcal{F}^{-1}(0)\big)\nonumber\\
  &\leq {\rm dist}\big(\mathcal{J}_{\gamma\mathcal{A}}(z\!-\!\gamma\mathcal{B}z), \mathcal{F}^{-1}(0)\big)\!+\!\|z-\mathcal{J}_{\gamma\mathcal{A}}(z-\gamma\mathcal{B}z)\|\nonumber\\
  &\leq \kappa{\rm dist}\big(0,\mathcal{F}(\mathcal{J}_{\gamma\mathcal{A}}(z-\gamma\mathcal{B}z))\big)
           +\|z-\mathcal{J}_{\gamma\mathcal{A}}(z-\gamma\mathcal{B}z)\|\nonumber\\
  &=\kappa{\rm dist}\big(-\mathcal{B}(\mathcal{J}_{\gamma\mathcal{A}}(z-\gamma\mathcal{B}z)),\mathcal{A}(\mathcal{J}_{\gamma\mathcal{A}}(z-\gamma\mathcal{B}z))\big)
           +\|z-\mathcal{J}_{\gamma\mathcal{A}}(z-\gamma\mathcal{B}z)\|\nonumber\\
  &\leq \kappa\|-\mathcal{B}(\mathcal{J}_{\gamma\mathcal{A}}(z-\gamma\mathcal{B}z))
             -\gamma^{-1}[z-\mathcal{J}_{\gamma\mathcal{A}}(z-\gamma\mathcal{B}z)]+\mathcal{B}z)\|
           +\|z-\mathcal{J}_{\gamma\mathcal{A}}(z-\gamma\mathcal{B}z)\|\nonumber\\
  &\leq (1+\gamma^{-1}\kappa+\beta^{-1}\kappa)\|z-\mathcal{J}_{\gamma\mathcal{A}}(z-\gamma\mathcal{B}z)\|
        =(1+\gamma^{-1}\kappa+\beta^{-1}\kappa)\|\mathcal{R}(z)\|
 \end{align*}
 which implies $\mathcal{R}$ is metric subregularity
 at $(z^*,0)\!\in\! {\rm gph}\,\mathcal{R}$. The proof is completed.
 \end{proof}

 By the above lemma and Theorem \ref{GDRS-convergence}, the generalized DRS algorithm
 is linear convergence if  $\mathcal{R}$ is metric subregular at ${z}^*$ for $0\in \mathcal{R}({z}^*)$.
 Moreover, When $\mathcal{B}$ is reduced to $\mathcal{B}=0$, strengthened as a cocoercive operator and
 specified as $\mathcal{B}=\mathcal{C}+\mathcal{D}$ with $\mathcal{C}$ being cocoercive and $\mathcal{D}$
 being single valued and lipshcitz, respectively.
 By Theorem \ref{theorem-GPPA} \ref{convergence-RFBS}, \ref{Davis-convergence} and above Lemma \ref{relation-AB},
 we know that the generalized PPA, the over-relaxed FBS algorithm,
 and the Davis-Yin's three operator splitting method are linearly convergent with
 $\mathcal{R}$ is metric subregular at ${z}^*$ for $0\!\in\!\mathcal{R}({z}^*)$ accordingly.
 Next, we give a sufficient condition to ensure the metric subregularirty of $\mathcal{R}$
 at ${z}^*$ for $0\!\in\!\mathcal{R}({z}^*)$.
 \begin{lemma}\label{projective-EB}
  The single-valued mapping $R$ is metrically subregular at ${z}^*$ for
   $0\in \mathcal{R}({z}^*)$ if the following projection type error bound \cite[Eq. 5]{Tseng95} holds
  \begin{equation}\label{projective-EB-eq1}
    {\rm dist}\big(z,\mathcal{R}^{-1}(0)\big) \leq \kappa''\|\mathcal{R}(z)\|,
     \quad\forall\ z\ {\rm with}\ \|\mathcal{R}(z)\| \leq \delta''
  \end{equation}
 \end{lemma}
 \begin{proof}
  Let $\delta'\!=\!\frac{\delta''}{2+\gamma\beta^{-1}}$. For any $z\!\in\! \mathbb{B}(z^*,\delta')$, using the equality
  $z^*\!=\!\mathcal{J}_{\gamma\mathcal{A}}(\mathcal{I}\!-\!\gamma\mathcal{B})^*$ yields
  \begin{equation}
  \|\mathcal{R}(z)\|
  =\|z-z^*-(\mathcal{J}_{\gamma\mathcal{A}}(\mathcal{I}-\gamma\mathcal{B})z-\mathcal{J}_{\gamma\mathcal{A}}(\mathcal{I}-\gamma\mathcal{B})z^*)\|
  \leq (2+\gamma\beta^{-1})\|z-z^*\|\leq \delta''
  \end{equation}
  Together with the above inequality and the projective bound \eqref{projective-EB}, we get the desired results
  that $R$ is metrically subregular at ${z}^*$ for $0\in \mathcal{R}({z}^*)$.
 \end{proof}

 Next, we give certain instances with the projection type error bound \eqref{projective-EB-eq1} or
 metric subregularity of $\mathcal{R}$ holding. Consequently, the metrically subregularity
 of $\mathcal{F}\!=\!\mathcal{A}\!+\!\mathcal{B}$ holds at ${z}^*$ for $0\in \mathcal{F}({z}^*)$. The proof of the following proposition
 is followed directly according to \cite{Tseng95,ZhS15,ZhZhS15}. Here, we omit the details.
\begin{proposition}\label{example}
Let $\mathcal{F}:=\mathcal{A}+\mathcal{B}$ where $\mathcal{A}$ is maximal monotone and
   $\mathcal{B}$ is single valued and Lipschitz continuous. Then, the operator $\mathcal{F}$ is metrically subregular
   at a point ${z}^*$ for $0\in \mathcal{F}({z}^*)$, i.e., there exists $\kappa,\delta$ such that inequality \eqref{temp-equa0} holds
   whenever one of the following statements holds.
 \begin{description}
   \item[(C1)]$\mathcal{F}:=\mathcal{A}+\mathcal{B}$ is strongly monotone;
   \item[(C2)]$\mathcal{B}$ is affine operator and $\mathcal{A}$ is polyhedron operator;
   \item[(C3)]$\mathcal{B}=\mathcal{E}\nabla f(\mathcal{E}\cdot)+C$ where $f$ is strongly convex and gradient
              Lipschitz $\mathcal{E}$ is linear operator and $C$ is a constant. $\mathcal{A}$ is the subdifferential operator
              of $\ell_{p}$ norm with $p\in [1,2]\bigcup \{\infty\}$ or  polyhedral convex function;
   \item[(C4)]$\mathcal{B}=\mathcal{E}\nabla f(\mathcal{E}\cdot)+C$ where $f$ is strongly convex and gradient
              Lipschitz and  $\mathcal{E}$ is linear operator and $C$ is a constant. $\mathcal{A}$ is the subdifferential operator
              of the nuclear norm. In addition, $-\mathcal{B}x \in {\rm ri}(\partial\|x\|_*)$.
 \end{description}
\end{proposition}

  To end this subsection, we make some comments on the metric subregularity of $\mathcal{S}_{\gamma,\mathcal{A},\mathcal{B}}$
  at a point $x^*$ with $0\in \mathcal{S}_{\gamma,\mathcal{A},\mathcal{B}}(x^*)$.
  Up to now, we are not clear whether the metric subregularity of $\mathcal{S}_{\gamma,\mathcal{A},\mathcal{B}}$
  at $(x^*,0) \in {\rm gph}\,(\mathcal{S}_{\gamma,\mathcal{A},\mathcal{B}})$ is weaker than that of
  $\mathcal{F}\!:=\!\mathcal{A}\!+\!\mathcal{B}$ at $(z^*,0)\in{\rm gph}\,\mathcal{F}$ or not
  when $\mathcal{A}$ or $\mathcal{B}$ is single-valued and Lipschitz continuous.
  The following proposition gives a sufficient condition to guarantee
  the  metric subregularity of $\mathcal{S}_{\gamma,\mathcal{A},\mathcal{B}}$
  at a point $(x^*,0) \in {\rm gph}\,(\mathcal{S}_{\gamma,\mathcal{A},\mathcal{B}})$.
  Its proof is also provided in the appendix.
  \begin{proposition}\label{prop4.1}
   If $\mathcal{F}$ is strongly monotone with constant $\alpha>0$ and one of $\mathcal{A}$ and
   $\mathcal{B}$ is single-valued and Lipschitz continuous with modulus $\beta>0$, then
   \begin{equation}\label{GDRS1-result}
     \mathcal{S}_{\gamma,\mathcal{A},\mathcal{B}}^{-1}(w)\subseteq\mathcal{S}_{\gamma,\mathcal{A},\mathcal{B}}^{-1}(0)
     + \kappa\|w\|\mathbb{B}\quad{\rm for}\ w\in\mathbb{X},
   \end{equation}
   This implies that $\mathcal{S}_{\gamma,\mathcal{A},\mathcal{B}}$
   is metrically subregular at $(x^*,0)\in{\rm gph}\,(\mathcal{S}_{\gamma,\mathcal{A},\mathcal{B}})$.
  \end{proposition}
 \section{Toy examples}
 In this section, we first consider the following nonsmooth convex optimization problems
 \begin{align}\label{convex-prob}
 \min f(x)+g(\mathcal{D}x)
 \end{align}
 where $f\!:\!\mathbb{Z}\to [-\infty,+\infty]$ and $g\!:\!\mathbb{Y}\to [-\infty,+\infty]$
 are low semicontinuous convex function, and $\mathcal{D}:\mathbb{Z}\to\mathbb{Y}$ is linear operator.
 Notice that the problem (\ref{convex-prob}) embodies an abundance of popular applications such as
 famous Rudin-Osher-Fatemi (ROF) denoising model \cite{ROF92}, $TVL_1$ minimization model \cite{CE05},
 the convex image segmentation model \cite{CEN06,GBO10} and the $\ell_1/\ell_1$-regularization model \cite{CTY13}.
 It is obvious that any optimal solution of \eqref{convex-prob} satisfies the inclusion
 $0\in \partial f(x)+\mathcal{D}^{*}\partial g(\mathcal{D}x)$.
 Involved in the dual variable $y\in \mathbb{Y}$, the above inclusion can be reformulated as the inclusion
 $(0,0) \in (\mathcal{T}_1+\mathcal{T}_2)(x,y)$ with
 \begin{align}\label{KKT-includ}
  \left\{
     \begin{array}{ll}
       \mathcal{T}_1(x,y) = (\mathcal{D}^{*}y, -\mathcal{D}x) & \hbox{} \\
       \mathcal{T}_2(x,y) = (\partial f(x),\partial g^*(y)) & \hbox{}
     \end{array}
   \right.
 \end{align}
 It is obvious that $\mathcal{T}_1$ is single valued, Lipschitz continuous and affine operator.
 Hence, we can apply the generalized DRS algorithm for the
 inclusion \eqref{KKT-includ} with following iterations
 \begin{subnumcases}{}\label{GDRS-convex1}
   x_{1}^{k} = (\mathcal{I}+\gamma^2\mathcal{D}^*\mathcal{D})^{-1}(z_{1}^{k}-\gamma\mathcal{D}^*z_{2}^{k})\\
   \label{GDRS-convex2}
   y_{1}^{k} = (\mathcal{I}+\gamma^2\mathcal{D}\mathcal{D}^*)^{-1}(\gamma\mathcal{D}z_{1}^{k}+z_{2}^{k})\\
   \label{GDRS-convex3}
   x_2^{k} = \min f(x_2)+\|x_2-(2x^{k}_1-z^{k}_1)\|^2/(2\gamma) \\
   \label{GDRS-convex4}
   y_2^{k} = \min g^*(y_2)+\|y_2-(2y^{k}_1-z^{k}_2)\|^2/(2\gamma)\\
   \label{GDRS-convex5}
   z_1^{k+1} = z_1^{k}+\lambda_k(x_2^{k}-x_1^{k}) \\
   \label{GDRS-convex6}
   z_2^{k+1} = z_2^{k}+\lambda_k(y_2^{k}-y_1^{k})
  \end{subnumcases}
  with $\gamma>0$ and relaxation parameter $\lambda_{k}\in (0,2)$.
  In the following, we specify model \eqref{convex-prob} as the $\ell_1/\ell_1$-regularization model proposed by Chan et al. \cite{CTY13}
  with the following form
  \begin{align}\label{L1L1}
  \min_{x}\, \|Ax-b\|_1 + \lambda\|x\|_1
  \end{align}
  \[
    \min_{x}\, \|Ax-b\|^2 + \lambda\|x\|_1
  \]
   Now, we apply the above generalized DRS algorithm \eqref{GDRS-convex1}-\eqref{GDRS-convex6} to the
   above $\ell_1/\ell_1$ regularization minimization \eqref{L1L1}. In this case,  we get that
  $\mathcal{T}_1(x,y) = (A^Ty, b-Ax)$ is singled valued and Lipschitz continuous operator and
  $\mathcal{T}_2(x,y) = (\partial \|x\|_{1}, N_{\|y\|_{\infty}\leq 1}(y))$ is polyhedral operator.
   By condition (C2) in Proposition \ref{example}, we know that the algorithm \eqref{GDRS-convex1}-\eqref{GDRS-convex6}
   converges linearly when it is applied to the $\ell_1/\ell_1$ regularization minimization \eqref{L1L1}.
   In the following, we verify this linear convergence result by the $\ell_1/\ell_1$-regularization
   minimization with random generated sensing matrix $A$ and regularization parameter $\lambda = 1$. The Figure \ref{figure-l1l1} shows
   that numerical performance of generalized DRS algorithm \eqref{GDRS-convex1}-\eqref{GDRS-convex6}
   when it is applied to the problem \eqref{L1L1}, which is coincided with Theorem \ref{GDRS-convergence}
   that the prime-dual points sequences $\{(x^k,y^k)\}$ and sequences $\{(z_1^k,z_2^k)\}$converge linearly.
   \begin{figure}
   \begin{minipage}{0.5\linewidth}
   \centering
   \includegraphics[width=3.1in]{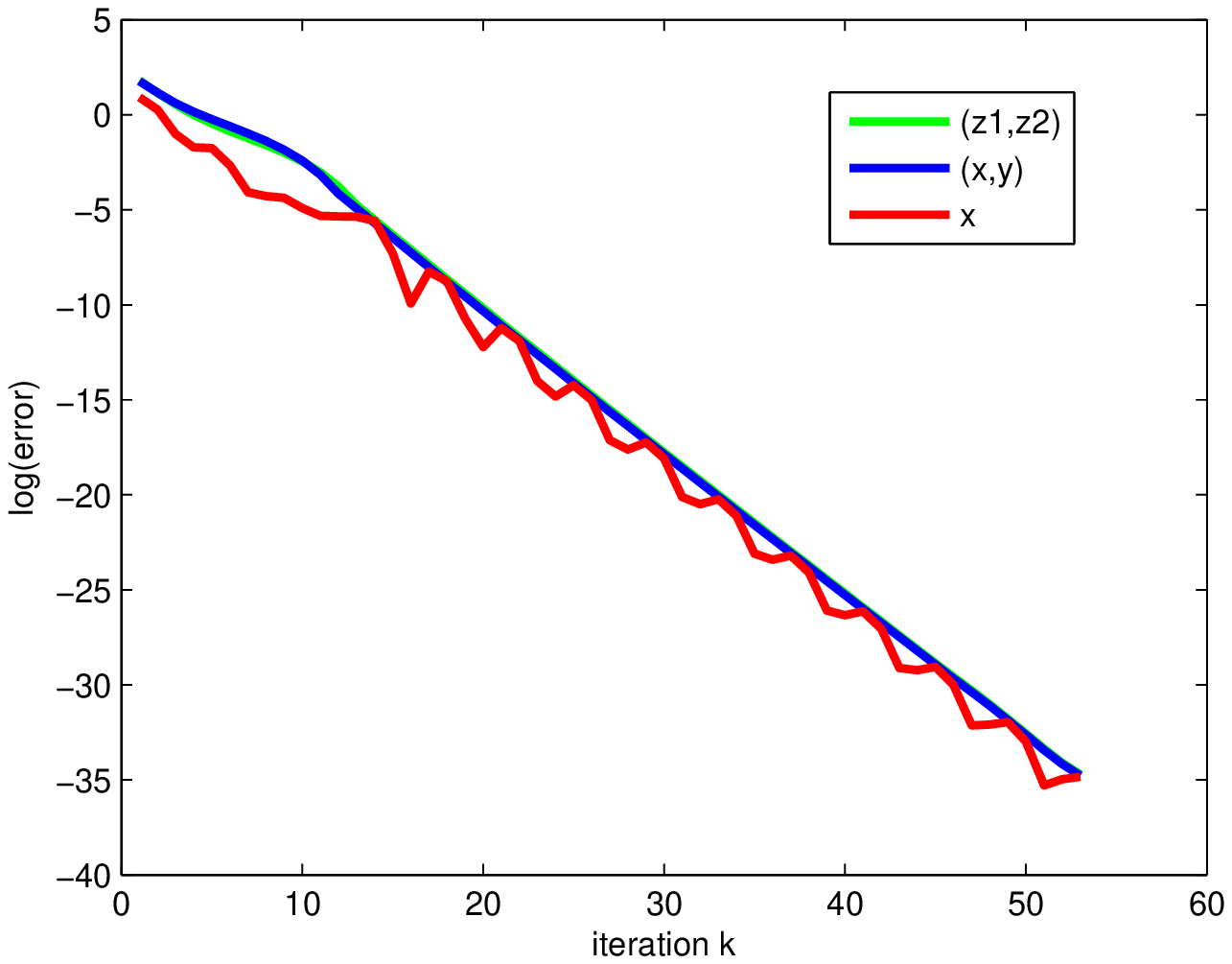}
   \caption{\small $\ell_1/\ell_1$-regularization minimization}
   \label{figure-l1l1}
   \end{minipage}
   \begin{minipage}{0.5\linewidth}
   \centering
   \includegraphics[width=3.1in]{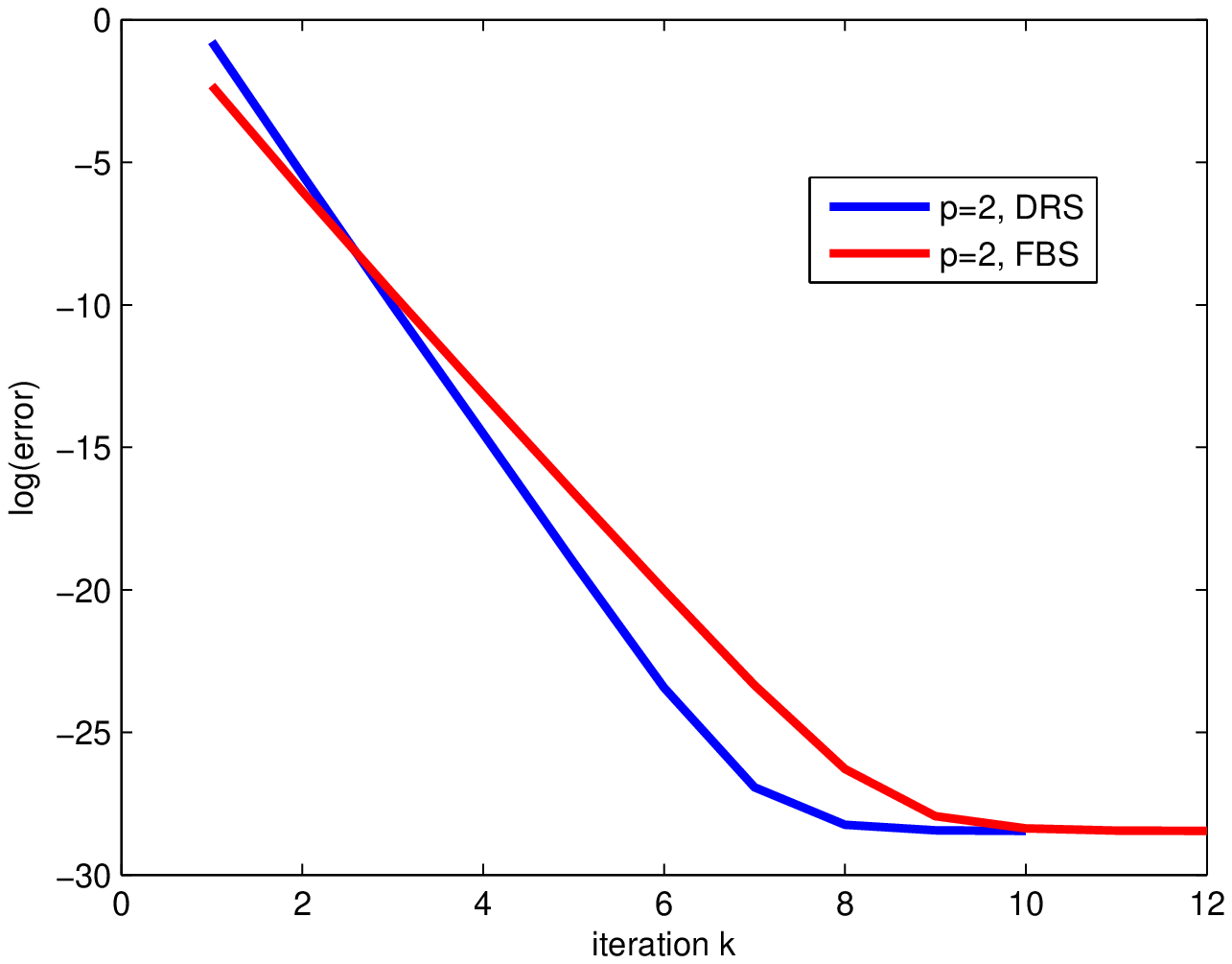}
   \caption{\small $\ell_p$-regularization minimization}
   \label{figure-lp}
   \end{minipage}
   \end{figure}

 Next, we consider the $\ell_{p}$ norm regularization problem with $p\in [1,2]\bigcup\{+\infty\}$.
 \begin{align}\label{convex-lp}
 \min f(\mathcal{D}x)+g(x)
 \end{align}
 where $f\!:\!\mathbb{Y}\to[-\infty,+\infty]$ is a strongly convex and gradient Lipschitz continuous function and
 $\mathcal{D}\!:\!\mathbb{X}\to\mathbb{Y}$ is a linear operator, and $g(x)=\sum_{J\in \mathcal{J}}w_{J}\|x_{J}\|_{p}$
 where $\mathcal{J}$ is a non-overlapping divisibility of index set $\{1,2,\cdots,n\}$ and
 $w_{J}>0$, and $\|x_{J}\|_{p}$ denotes the $l_{p}$ norm defined by $\|x_{J}\|_{p} = (\sum_{i=1}^{J}\|x_{i}\|^{p})^{\frac{1}{p}}$ .
 This problem \eqref{convex-lp} incorporates massive applications such as
 Group-lasso regularization \cite{YL06}, $\ell_{1,p}$-regularization regression \cite{FR08,EKB10} and the referees in \cite{ZhZhS15}.
 Here, we consider the following $\ell_{p}$ norm regularization problem
  \begin{equation}\label{easy-lp}
 \min_{x}\, \frac{1}{2}\|Ax-b\|^2+\lambda\|x\|_{p}
 \end{equation}
 Now, we apply the over-relaxed forward backward splitting algorithm \eqref{RFBS} and the generalized
 Douglas Rachford splitting algorithm \eqref{GDRS} to the problem \eqref{easy-lp}
 with random generated sensing matrix $A$, regularized parameter $\lambda = 1$ and
 $p=2$, respectively. The figure \ref{figure-lp} shows
 performance of over-relaxed FBS and generalized DRS algorithm
 when they are applied to the toy example \eqref{easy-lp}, which is
 coincided with Theorem \ref{convergence-RFBS}, \ref{GDRS-convergence}.

%
  \section{Conclusion}

  In this paper, for the inclusion problem \eqref{composite-inclusion}, we have established
  the linear convergence rate of the generalized PPA and several popular splitting algorithms
  under the metric subregularity of the composite operator, which is much weaker than
  the existing ones that almost all require the strong monotonicity of the composite operator.
  Some sufficient condition are provided to ensure the metric subregularity of the composite operator $\mathcal{F}$ holds
  at a point $(z^*,0)\in {\rm gph}\,\mathcal{F}$. The preliminary numerical performances also
  support the theoretical results.

 \bigskip
 \noindent
 {\bf\large Appendix}

 \medskip
 \noindent
{\bf Proof of Lemma \ref{zero-point1}: }
  Part (a) directly follows from \cite[Proposition 25.1(ii)]{BC11}.
  We next make use of $\mathcal{S}_{\gamma,\mathcal{A},\mathcal{B}}$ in \cite{EB92}
   to prove the inclusion of part (b), where $\mathcal{S}_{\gamma,\mathcal{A},\mathcal{B}}$ is defined by
   \begin{equation}\label{Soperator}
    \mathcal{S}_{\gamma,\mathcal{A},\mathcal{B}}=\Big\{(v+\gamma b,u-v)\ |\ (u,b)\in\mathcal{B},\,(v,a)\in\mathcal{A},\,v+\gamma a=u-\gamma b\Big\}.
   \end{equation}
   From \cite[Theorem 5]{EB92}, it follows that
   \(
    {\rm Fix}\mathcal{T}=\mathcal{S}_{\gamma,\mathcal{A},\mathcal{B}}^{-1}(0)=\bigcup_{u\in\mathbb{X}}\big[u+\gamma(-\mathcal{A}u\cap\mathcal{B}u)\big].
   \)
   Let $x$ be an arbitrary point from ${\rm Fix}\mathcal{T}$. Then,
   there exist $u\in\mathbb{X}$ and $b\in (-\mathcal{A}u\cap\mathcal{B}u)$ such that
   $x=u+\gamma b\in(\mathcal{I}-\gamma\mathcal{A})(u)$. Clearly, $0\in\mathcal{A}u+\mathcal{B}u=\mathcal{F}u$,
   i.e., $u\in\mathcal{F}^{-1}(0)$. Hence,
   \[
     x\in{\textstyle\bigcup_{z\in\mathcal{F}^{-1}(0)}}(\mathcal{I}-\gamma\mathcal{A})z=(\mathcal{I}-\gamma\mathcal{A})(\mathcal{F}^{-1}(0)).
   \]
   The inclusion then follows from the arbitrariness of $x$ in the set ${\rm Fix}\mathcal{T}$.

   \medskip

   Now assume that $\mathcal{A}$ is single-valued. To establish the equality, it suffices to argue that
   $(\mathcal{I}-\gamma\mathcal{A})(\mathcal{F}^{-1}(0))\subseteq{\rm Fix}\mathcal{T}$.
   Let $x$ be an arbitrary point from $(\mathcal{I}-\gamma\mathcal{A})(\mathcal{F}^{-1}(0))$. Then there exist
   $z\in\mathcal{F}^{-1}(0)$ and $y=\mathcal{A}z$ such that $x=z-\gamma y$.
   Since $0\in\mathcal{A}z+\mathcal{B}z$, we have $-\mathcal{A}z=-y\in\mathcal{B}z$.
   Thus, $x\in z+\gamma\mathcal{B}z$, which by $z\in\mathcal{F}^{-1}(0)$ and part (a)
   implies that $x\in{\rm Fix}\mathcal{T}$. The inclusion follows by
   the arbitrariness of $x$ in $(\mathcal{I}-\gamma\mathcal{A})(\mathcal{F}^{-1}(0))$.
   Hence, the proof is completed.
   \qquad\qquad\qquad\qquad\qquad\qquad\qquad\qquad\qquad\qquad\qquad\qquad\qquad\qquad\quad $\Box$

 \bigskip

  \bigskip

 \noindent
{\bf Proof of Lemma \ref{zero-point2}: }
  Part (a) follows from \cite[Lemma 3.2]{Davis15}. It suffices to prove part (b).
  From \cite[Lemma 3.2]{Davis15}, it follows that
   \(
    {\rm Fix}\mathcal{T}=\bigcup_{u\in\mathcal{F}^{-1}(0)}\big[u+\gamma((-\mathcal{A}u-\mathcal{C}u)\cap\mathcal{B}u)\big].
   \)
   Let $x$ be an arbitrary point from ${\rm Fix}\mathcal{T}$. Then, there exist
   $u\in\mathcal{F}^{-1}(0)$ and $b\in(-\mathcal{A}u-\mathcal{C}u)\cap\mathcal{B}u$ such that
   $x=u+\gamma b\in\big[\mathcal{I}-\gamma(\mathcal{A}+\mathcal{C})\big]u$.
   This immediately implies that
   \[
     x\in\bigcup_{z\in\mathcal{F}^{-1}(0)}\big[\mathcal{I}-\gamma(\mathcal{A}+\mathcal{C})\big]z
     =\big[\mathcal{I}-\gamma(\mathcal{A}+\mathcal{C})\big](\mathcal{F}^{-1}(0)).
   \]
   By the arbitrariness of $x$ in ${\rm Fix}\mathcal{T}$, the inclusion follows.
   Now assume that $\mathcal{A}$ is single-valued. We only need to argue that
   $\big[\mathcal{I}-\gamma(\mathcal{A}+\mathcal{C})\big](\mathcal{F}^{-1}(0))\subseteq{\rm Fix}\mathcal{T}$.
    Let $x$ be an arbitrary point
   from $\big[\mathcal{I}-\gamma(\mathcal{A}+\mathcal{C})\big](\mathcal{F}^{-1}(0))$.
   Then there exist $z\in\mathcal{F}^{-1}(0)$ and $y=(\mathcal{A}+\mathcal{C})z$ such that $x=z-\gamma y$.
   Since $0\in\mathcal{A}z+\mathcal{B}z+\mathcal{C}z$, we have $-(\mathcal{A}+\mathcal{C})z=-y\in\mathcal{B}z$.
   Thus, $x\in z+\gamma\mathcal{B}z$. This, along with $z\in\mathcal{F}^{-1}(0)$ and part (a),
   implies that $x\in{\rm Fix}\mathcal{T}$. By the arbitrariness of $x$ in
   $\big[\mathcal{I}-\gamma(\mathcal{A}+\mathcal{C})\big](\mathcal{F}^{-1}(0))$,
   the inclusion follows.  \qquad\qquad\qquad\qquad $\Box$

 \bigskip

  \bigskip

 \noindent
{\bf Proof of Proposition \ref{prop4.1}: }
   Let $x$ be an arbitrary point from ${\rm dom}\,\mathcal{S}_{\gamma,\mathcal{A},\mathcal{B}}$.
   Since $\mathcal{S}_{\gamma,\mathcal{A},\mathcal{B}}$ is maximal monotone (see \cite[Theorem 4]{EB92}),
   the set $\mathcal{S}_{\gamma,\mathcal{A},\mathcal{B}}(x)$ is closed convex by \cite[Exercise 12.8]{RW98}.
   In the following, we proceed the arguments by two cases as shown below.

   \medskip
   \noindent
   {\bf Case 1: $\mathcal{A}$ is single-valued and Lipschitz continuous with modulus $\beta>0$}.
   From the closed convexity of $\mathcal{S}_{\gamma,\mathcal{A},\mathcal{B}}(x)$ and
   the expression of $\mathcal{S}_{\gamma,\mathcal{A},\mathcal{B}}$,
   there exist $u,b,v\in\mathbb{X}$ with $(u,b)\in\mathcal{B}$ and $v+\gamma \mathcal{A}v=u-\gamma b$
   such that $u-v={\rm dist}(0,\mathcal{S}_{\gamma,\mathcal{A},\mathcal{B}}(x))$ for $x=v+\gamma b$. So,
   \begin{equation}\label{GDRS1-ineq1}
    \gamma^{-1}(u-v)+(\mathcal{A}u-\mathcal{A}v)\in\mathcal{F}u.
   \end{equation}
   By Lemma \ref{zero-point1}(b) and the remark after it, we have
   \(
     \mathcal{S}_{\gamma,\mathcal{A},\mathcal{B}}^{-1}(0)=(\mathcal{I}-\gamma\mathcal{A})(\mathcal{F}^{-1}(0)).
   \)
   Thus,
   \begin{align}\label{GDRS1-ineq2}
    {\rm dist}(x,\mathcal{S}_{\gamma,\mathcal{A},\mathcal{B}}^{-1}(0))
    &={\rm dist}(u-\gamma\mathcal{A}v,\mathcal{S}_{\gamma,\mathcal{A},\mathcal{B}}^{-1}(0))
    \le \|u-\gamma\mathcal{A}v-z^*+\gamma\mathcal{A}z^*\|\nonumber\\
    &\le\|u-z^*\|+\gamma \beta \|v-z^*\|\le (1+\gamma \beta)\|u-z^*\|+\gamma\beta\|u-v\|,
   \end{align}
   where $z^*$ is an arbitrary point from $\mathcal{F}^{-1}(0)$.
   Since $\mathcal{F}=\mathcal{A}+\mathcal{B}$ is strongly monotone with constant $\alpha$,
   $\mathcal{F}^{-1}$ is single-valued and Lipschitz continuous with modulus $\alpha^{-1}$, i.e.,
   \begin{equation}\label{strong-m}
     \|\mathcal{F}^{-1}(w)-\mathcal{F}^{-1}(0)\|\le \alpha^{-1}\|w\|\quad\ \forall w\in\mathbb{X}.
   \end{equation}
   Let $\overline{w}\in \mathcal{F}u$ be such that $\|\overline{w}\|={\rm dist}(0,\mathcal{F}u)$.
   Then, by the last inequality, it follows that
   \[
     \|u-z^*\|=\|\mathcal{F}^{-1}(\overline{w})-\mathcal{F}^{-1}(0)\|\le \alpha^{-1}\|\overline{w}\|
     =\alpha^{-1}{\rm dist}(0,\mathcal{F}u).
   \]
   Combining this inequality with inequalities \eqref{GDRS1-ineq1} and \eqref{GDRS1-ineq2}, we have that
   \begin{align*}
   {\rm dist}(x,\mathcal{S}_{\gamma,\mathcal{A},\mathcal{B}}^{-1}(0))
   &\le \alpha^{-1}(1+\gamma \beta){\rm dist}(0,\mathcal{F}u)+\gamma \beta\|u-v\|\\
   &\le  \alpha^{-1}(1+\gamma\beta)\big\|\gamma^{-1}(u-v)+(\mathcal{A}u-\mathcal{A}v)\big\|+\gamma \beta\|u-v\|\\
   &\le \big[\alpha^{-1}(1+\gamma\beta)(\gamma^{-1}+\beta)+\gamma \beta\big]\|u-v\|\\
   &=\big[\alpha^{-1}(1+\gamma\beta)(\gamma^{-1}+\beta)+\gamma\beta\big]{\rm dist}(0,\mathcal{S}_{\gamma,\mathcal{A},\mathcal{B}}(x)).
   \end{align*}
   This, along with the arbitrariness  of $x$ in ${\rm dom}\,\mathcal{S}_{\gamma,\mathcal{A},\mathcal{B}}$,
   is equivalent to saying that
   \[
    \mathcal{S}_{\gamma,\mathcal{A},\mathcal{B}}^{-1}(w)\subseteq\mathcal{S}_{\gamma,\mathcal{A},\mathcal{B}}^{-1}(0)
     + \kappa\|w\|\mathbb{B}\quad{\rm for}\ w\in\mathbb{X}.
   \]
   Notice that $\mathcal{S}_{\gamma,\mathcal{A},\mathcal{B}}^{-1}(0)$ is singleton
   due to the strong monotonicity of $\mathcal{F}$. Together with the last inclusion,
   $\mathcal{S}_{\gamma,\mathcal{A},\mathcal{B}}$
   is metrically subregular at $(x^*,0)\in{\rm gph}\mathcal{S}_{\gamma,\mathcal{A},\mathcal{B}}$.

   \medskip
   \noindent
   {\bf Case 2: $\mathcal{B}$ is single-valued and Lipschitz continuous with modulus $\beta>0$}.
   From the closed convexity of $\mathcal{S}_{\gamma,\mathcal{A},\mathcal{B}}(x)$ and
   the expression of $\mathcal{S}_{\gamma,\mathcal{A},\mathcal{B}}$,
   there exist $u,v,a\in\mathbb{X}$ with $(v,a)\in\mathcal{A}$ and $v+\gamma a=u\!-\!\gamma\mathcal{B}u$
   such that $u-v={\rm dist}(0,\mathcal{S}_{\gamma,\mathcal{A},\mathcal{B}}(x))$ for $x=v\!+\!\gamma\mathcal{B}u$.
   So,
   \begin{equation}\label{GDRS1-ineq3}
    \gamma^{-1}(u-v)+(\mathcal{B}v-\mathcal{B}u)\in\mathcal{F}v.
   \end{equation}
   By Lemma \ref{zero-point1}(a) and the remark after it, we have
   \(
     \mathcal{S}_{\gamma,\mathcal{A},\mathcal{B}}^{-1}(0)=(\mathcal{I}+\gamma\mathcal{B})(\mathcal{F}^{-1}(0)).
   \)
   Thus,
   \begin{align}\label{GDRS1-ineq4}
    {\rm dist}(x,\mathcal{S}_{\gamma,\mathcal{A},\mathcal{B}}^{-1}(0))
    &={\rm dist}(v+\gamma\mathcal{B}u,\mathcal{S}_{\gamma,\mathcal{A},\mathcal{B}}^{-1}(0))
    \le \|v+\gamma\mathcal{B}u-z^*-\mathcal{B}z^*\|\nonumber\\
    &\le\|v-z^*\|+\gamma \beta \|u-z^*\|\nonumber\\
    &\le (1+\gamma\beta)\|v-z^*\|+\gamma \beta\|u-v\|,
   \end{align}
   where $z^*$ is an arbitrary point from $\mathcal{F}^{-1}(0)$.
   Since $\mathcal{F}=\mathcal{A}+\mathcal{B}$ is strongly monotone with constant $\alpha$,
   $\mathcal{F}^{-1}$ is single-valued and Lipschitz continuous with modulus $\alpha^{-1}$.
   Let $\overline{w}\in\mathcal{F}v$ be such that $\|\overline{w}\|={\rm dist}(0,\mathcal{F}v)$.
   Then, from \eqref{strong-m} it follows that
   \[
     \|v-z^*\|=\|\mathcal{F}^{-1}(\overline{w})-\mathcal{F}^{-1}(0)\|\le \alpha^{-1}\|\overline{w}\|
     =\alpha^{-1}{\rm dist}(0,\mathcal{F}v).
   \]
   Combining this inequality with inequalities \eqref{GDRS1-ineq3} and \eqref{GDRS1-ineq4}, we have that
   \begin{align*}
   {\rm dist}(x,\mathcal{S}_{\gamma,\mathcal{A},\mathcal{B}}^{-1}(0))
   &\le \alpha^{-1}(1+\gamma \beta){\rm dist}(0,\mathcal{F}v)+\gamma\beta\|u-v\|\\
   &\le  \alpha^{-1}(1+\gamma \beta)\big\|\gamma^{-1}(u-v)+(\mathcal{B}v-\mathcal{B}u)\big\|+\gamma \beta\|u-v\|\\
   &\le \big[\alpha^{-1}(1+\gamma \beta)(\gamma^{-1}+\beta)+\gamma \beta\big]\|u-v\|\\
   &=\big[\alpha^{-1}(1+\gamma \beta)(\gamma^{-1}+\beta)+\gamma\beta\big]{\rm dist}(0,\mathcal{S}_{\gamma,\mathcal{A},\mathcal{B}}(x)).
   \end{align*}
   The rest arguments are similar to those for Case 1, and we omit them. \qquad\qquad\qquad $\Box$

\end{document}